\documentclass[11pt, twoside, leqno]{amsart}  

\usepackage{a4wide}

\usepackage{lipsum}
\usepackage{amsfonts}
\usepackage{graphicx}
\usepackage{epstopdf}
\usepackage{fancyhdr}
\usepackage[normalem]{ulem}
\usepackage{cancel}
\usepackage{calligra}
\usepackage{amsfonts,amsmath,amsthm,amssymb}
\usepackage{mathtools}
\usepackage{hyperref}
\usepackage{autonum}
\usepackage{hhline}
\usepackage{array}
\usepackage{bibentry}
\usepackage{pifont}
\newcommand{\cmark}{\ding{51}}%
\usepackage{diagbox}
\usepackage[table,x11names]{xcolor}
\usepackage{tcolorbox}
\usepackage{mdframed}
\usepackage{multicol}
\usepackage{graphicx}
\usepackage{subcaption}
\usepackage{moreverb}
\usepackage{bbm}

\usepackage{todonotes}
\usepackage{scalerel,amssymb}
\allowdisplaybreaks
\usepackage{mathrsfs}  
\usepackage{lineno}
\usepackage{todonotes}
\usepackage{tikz}
\usepackage{pgfplots}
\usetikzlibrary{arrows.meta}
\usepackage[numbers,sort&compress]{natbib}
\definecolor{mygreen}{HTML}{43a047}
\usepackage{doi}
\usepackage{adjustbox}
\usepackage{alphalph}
\usepackage{booktabs}
\usepackage{makecell}
\usepackage{appendix}
\usepackage{accents}
\usepackage[normalem]{ulem}

\def\eps{\varepsilon}
\def\beps{\bar{\varepsilon}}

\def\calM{\mathcal{M}}
\def\fB{F^\textbf{B}}

\def\ulal{\underline{\mathfrak{m}}}
\def\olal{\overline{\mathfrak{m}}}

\def\opsi{\overline{\psi}}
\def\ophi{\overline{\phi}}

\newcommand{\Om}{\Omega}
\newcommand{\D}{\Delta}

\newcommand{\Dt}{\textup{D}_t}

\def\aaa{\mathfrak{m}}

\def\bbb{\mathfrak{n}}
\def\lll{\mathfrak{l}}

\def\tk{k}

\definecolor{darkgreen}{rgb}{0,0.5,0}

\newcommand{\ueps}{u^\varepsilon}
\newcommand{\uteps}{u_t^\varepsilon}

\newcommand{\cAone}{C_{\textbf{A}_1}}
\newcommand{\cAtwo}{C_{\textbf{A}_2}}
\newcommand{\cAthree}{C_{\textbf{A}_3}}

\newcommand{\pt}{p_t}

\newcommand{\ptt}{p_{tt}}

\newcommand{\pepstt}{\psi_{tt}^\eps}

\newcommand{\pepst}{\psi_{t}^\eps}
\newcommand{\peps}{\psi^\eps}
\newcommand{\ddt}{\frac{\textup{d}}{\textup{d}t}}

\newcommand{\Dal}{{\textup{D}}_t^\alpha}
\newcommand{\Doal}{{\textup{D}}_t^{1-\alpha}}

\newcommand{\dt}{\, \textup{d} t}
\newcommand{\ds}{\, \textup{d} s }

\newcommand{\dxs}{\, \textup{d}x\textup{d}s}

\newcommand{\intO}{\int_{\Omega}}
\newcommand{\nLtwo}[1]{\|#1\|_{L^2(\Omega)}}

\newcommand{\R}{\mathbb{R}} 
 
\newcommand{\Ltwo}{L^2(\Omega)}

\newcommand{\Linf}{L^\infty(\Omega)}
\newcommand{\Hone}{H^1(\Omega)}

\newcommand{\Hthree}{H^3(\Omega)}
\newcommand{\Honezero}{H_0^1(\Omega)}
\newcommand{\Honetwo}{{H_\diamondsuit^2(\Omega)}}
\newcommand{\Honethree}{{H_\diamondsuit^3(\Omega)}}
\newcommand{\Honefour}{{H_\diamondsuit^4(\Omega)}}

\newtheorem{theorem}{Theorem}

\newtheorem{proposition}{Proposition}
\newtheorem{corollary}[theorem]{Corollary}
\newtheorem*{assumption*}{Assumptions}

\numberwithin{lemma}{section}
\numberwithin{proposition}{section}
\numberwithin{theorem}{section}
\numberwithin{equation}{section}
\makeatletter
\newcommand{\leqnomode}{\tagsleft@true}
\newcommand{\reqnomode}{\tagsleft@false}
\makeatother
\newcommand{\bfq}{\boldsymbol{q}}

\newcommand{\calB}{\mathcal{B}}
\newcommand{\frakKeps}{\mathfrak{K}_\varepsilon}
\newcommand{\frakK}{\mathfrak{K}}

\newcommand\Lconv{\ast}

\newcommand{\TK}{\mathcal{T}}

\newcommand{\rt}{\tau_\theta}

\definecolor{grey}{rgb}{0.5,0.5,0.5}

\newcommand{\sometimes}{
	\begin{tikzpicture}[scale=0.5, baseline=-1.2mm, thick]
		\draw[line width=1ex, line cap=round, gray]  (0,0) -- (4mm,0);
	\end{tikzpicture}
}

\title[Kuznetsov and Blackstock equations with nonlocal dissipation]{The Kuznetsov and Blackstock equations of nonlinear acoustics with nonlocal-in-time dissipation}


\author[B. Kaltenbacher, M. Meliani, and V. Nikoli\'{c}]{\small Barbara Kaltenbacher, Mostafa Meliani, and Vanja Nikoli\'{c}}
\address{  \small
	Department of Mathematics, 
	Alpen-Adria-Universit\"at Klagenfurt 
	\\ Universit\"atsstra\ss e 65--67, A-9020 Klagenfurt, Austria}
\email{barbara.kaltenbacher@aau.at}
\address{ 
	Department of Mathematics \\ 
	Radboud University   \\ 
	Heyendaalseweg 135,
	6525 AJ Nijmegen, The Netherlands}
\email{mostafa.meliani@ru.nl} 
\email{vanja.nikolic@ru.nl}

\nobibliography*
\begin{document}
\vspace*{4mm}
\begin{abstract}
	In ultrasonics, nonlocal quasilinear wave equations arise when taking into account a class of heat flux laws of Gurtin--Pipkin type within the system of governing equations of sound motion. 
 The present study extends previous work by the authors to incorporate nonlocal acoustic wave equations with quadratic gradient nonlinearities which require a new approach in the energy analysis.
 More precisely, we investigate the Kuznetsov and Blackstock equations with dissipation of fractional type and identify a minimal set of assumptions on the memory kernel needed for each equation. In particular, we discuss the physically relevant examples of Abel and Mittag-Leffler kernels. We perform the well-posedness analysis uniformly with respect to a small parameter on which the kernels depend and which can be interpreted as the sound diffusivity or the thermal relaxation time. We then analyze the limiting behavior of solutions with respect to this parameter, and how it is influenced by the specific class of memory kernels at hand. Through such a limiting study, we relate the considered nonlocal quasilinear equations to their limiting counterparts and establish the convergence rates of the respective solutions in the energy norm. 
\end{abstract}
\vspace*{-8mm}
\maketitle           
\section{Introduction} \label{Sec:Introduction}
The Kuznetsov~\cite{kuznetsov1971equations} and Blackstock~\cite{blackstock1963approximate} equations are classical models of nonlinear acoustics derived from the governing equations of fluid dynamics. In~\cite{kaltenbacher2023limting}, nonlocal-in-time analogues of these equations were derived by assuming a class of general heat exchange laws in the acoustic medium, so-called Gutin--Pipkin heat flux laws~\cite{gurtin1968general}.
\\\indent
In this work, we extend the analysis in \cite{kaltenbacher2023limting} in which a fractionally damped Westervelt equation (a cousin of Blackstock's and Kuznetsov's equations) in pressure form was studied. Here, we intend to cover the more involved case of quadratic gradient nonlinearities and as a special case retrieve a nonlocal Westervelt equation in potential form. We analyze, in a smooth bounded domains $\Om$ with homogeneous Dirichlet boundary conditions, the acoustic wave equation 
\begin{align} \label{abstract_wave_eq_Kuznetsov}
(1+2\tk\peps_t)\peps_{tt}-c^2 \Delta \peps - \frakKeps* \D\peps_{t} + 2\ell\; \nabla \peps\cdot \nabla \peps_t=0,
\end{align}
which we refer to as of Kuznetsov type. By setting $\ell=0$ and adjusting the medium-dependent constant $\tk$, we recover a nonlocal Westervelt equation in potential form.
\\ \indent We are also interested in the analysis of the Blacktock-type equation:
\begin{align} \label{abstract_wave_eq_Blackstock}
\peps_{tt}-c^2 (1-2\tk\peps_t)\Delta \peps - \frakKeps* \D\peps_{t} + 2\ell\; \nabla \peps\cdot \nabla \peps_t=0.
\end{align}
Above, $c>0$ is the speed of sound and $\tk$, a  nonlinearity parameter. Typically, $\ell\in\{0,1\}$ depending on the nonlinear acoustic model used, however, for the analysis, we can view it as a real number. The prime examples of the involved memory kernel $\frakKeps$ will be the Abel kernel: 
	\begin{equation} 
	\delta \rt^{-\alpha}\frac{1}{\Gamma(\alpha)} t^{\alpha-1}, \quad \alpha \in (0,1) 
	\end{equation}
and the Mittag-Leffler kernel
\begin{equation} 
	\begin{aligned}
		\delta \left(\frac{\rt}{\tau}\right)^{a-b}\frac{1}{\tau^b}t^{b-1}E_{a,b}\left(-\left(\frac{t}{\tau}\right)^a\right),
	\end{aligned}
	\end{equation}
with $0<a, b\leq1$. The physical parameters $\delta$ and $\tau$ are the sound diffusivity and thermal relaxtion time, respectively, and are typically small. Thus, $\eps$ on which the memory kernel $\frakKeps$ depends, and that we aim to send to zero, will be  $\eps=\delta$ in the case of the Abel kernel, whereas for the Mittag-Leffler kernels, we will have the two cases $\eps=\delta$ or $\eps = \tau$ (while
the respective other small parameter is positive fixed).
 \\
\indent
The difference between \eqref{abstract_wave_eq_Kuznetsov} and \eqref{abstract_wave_eq_Blackstock}, which resides in the position of the quasilinearity, leads to distinct assumptions and analyses when it comes to the uniform-in-$\eps$ well-posedness, therefore, we will study them separately in Sections~\ref{Sec:Kuznetsov_prop} and~\ref{Sec:Blackstock_prop}. 
\subsection*{State of the art} To the best of our knowledge, together with \cite{kaltenbacher2023limting}, this is the first body of rigorous work dealing with the analysis and singular limits of quasilinear wave equations with damping of nonlocal/time-fractional type.
The analysis of these equations is challenging because of the combination of the nonlinear evolution and the nonrestrictive assumptions imposed on the memory kernel. 
To show well-posedness of Kuznetsov's and Blackstock's equations, one must ensure that the quasilinear coefficient $1 + 2\tk \pepst$ (respectively $1 - 2\tk \pepst$) does not degenerate, uniformly in $\eps$. Note that the local well-posedness of the strongly damped Blackstock equation can be performed without the need  of a nondegeneracy condition by leaning on the parabolic structure of the equation; see e.g., \cite{nikolic2022time,fritz2018well}. However, since our goal in the analysis is not to rely on the strong damping so as to be able to generalize the results to weaker types of damping, i.e., fractional and to be able to take limits as the sound diffusivity parameter vanishes, we will require the coefficient $1 - 2\tk \pepst$ not degenerate.
\\
\indent We note that the well-posedness analysis of the nonlocal Westervelt equation in pressure form 
\begin{equation} \label{West_general}
	((1+2\tilde k\ueps)\uteps)_t-c^2 \Delta \ueps - \Delta \frakKeps * \uteps = 0
\end{equation} 
with fractional kernels can be found in~\cite{kaltenbacher2022inverse,baker2022numerical} and in a more general framework in \cite{kaltenbacher2023limting}. The unknown $\ueps$ is the acoustic pressure and $\tilde k$ is a medium dependent nonlinearity parameter related to $\tk$. Compared to the analysis in~\cite{kaltenbacher2023limting,kaltenbacher2022inverse, baker2022numerical}, the presence of a time-derivative ($\pepst$ instead of $\ueps$) in the nonlinear coefficient in \eqref{abstract_wave_eq_Kuznetsov} and \eqref{abstract_wave_eq_Blackstock} puts an additional strain on the analysis, as it requires obtaining $\eps$-uniform bounds on $\|\pepst\|_{L^\infty(\Linf)}$ and guaranteeing their smallness. We also need, in the case of the Kuznetsov and Blackstock settings, to extract enough regularity to absorb the quadratic gradient nonlinearity (that is, $2\ell \nabla \peps \cdot \nabla\pepst$). To achieve this goal, we will need to rely on an additional assumption on the kernel compared to~\cite{kaltenbacher2023limting} (see Assumption~\eqref{assumption2}). The well-posedness analysis will rely on considering a time-differentiated linearized PDE, and using a Banach's fixed-point argument.
\\
\indent 
The local-in-time counterparts of the acoustic models considered in the present work are by now well-studied; we refer to, e.g.,~\cite{fritz2018well, tani2017mathematical, kaltenbacher2022parabolic, mizohata1993global,dekkers2017cauchy,Wilke} and the references contained therein. 
\subsection*{Main results} 
Our contributions  are twofold. First, we establish the well-posedness of a family of nonlocal Blackstock and Kuznetsov equations supplemented with initial and homogeneous Dirichlet boundary conditions. 
Sufficient smoothness of initial data will be required, namely
\[
\psi_0, \psi_1 \in \{v \in H^4(\Om) \cap \Honezero: \Delta v \vert_{\partial \Omega}=0 \}.
\]
Furthermore, smallness of data and short final time will be needed to ensure $\eps$-uniform well-posedness of the nonlinear problem. However, smallness of the initial conditions will be imposed in a lower-order norm than that of their regularity space by relying on an Agmon's interpolation inequality~\cite[Lemma 13.2]{agmon2010lectures}. Theorems~\ref{Thm:Wellp_2ndorder_nonlocal_Kuzn} and \ref{Thm:Wellp_2ndorder_nonlocal_Blackstock} establish well-posedness of, among others, the fractionally damped Kuznetsov and Blackstock equations. 

Second, we identify the assumptions on the memory kernel under which one can take the limit $\eps\searrow0$ 
in \eqref{abstract_wave_eq_Kuznetsov} or \eqref{abstract_wave_eq_Blackstock}. Indeed, provided $\frakKeps$ verifies certain nonrestrictive assumptions and converges in the $\|\cdot \Lconv \ 1\|_{L^1(0,T)}$ norm to some measure $\frakK_0$ as $\eps\searrow0$, we show that the solutions converge in the standard energy norm 
\begin{equation}\label{energy_norm}
	\|\psi\|_{\textup{E}} := \left(\|\psi_t\|^2_{L^\infty(\Ltwo)}+ \|\psi\|^2_{L^\infty(\Hone)}\right)^{1/2}
\end{equation}
to the solution of the models with $\frakK_0$ as memory kernel. To establish the rate of convergence, we need to consider specific examples. We will focus then on the aforementioned Abel and Mittag-Leffler kernels. The vanishing limits are detailed in Section~\ref{Sec:limits_Kuzn} and \ref{Sec:nonlin_wellp_cont_Black}.
These results significantly generalize~\cite[Theorem 7.1]{kaltenbacher2022parabolic}, where the vanishing sound diffusivity limit of the strongly damped Kuznetsov equation (obtained here by setting $\frakKeps=\eps \delta_0$ in \eqref{abstract_wave_eq_Kuznetsov}) has been studied. 

The conditions on the kernels for the Blackstock case are weaker than those for the Kuznetsov case (\emph{cf}. \eqref{assumption3} versus \eqref{assumption3_Black}).

\subsection*{Organization of the paper} The rest of the present paper is organized as follows. In Section~\ref{Sec:AcousticModeling}, we motivate this study by 
recalling Kuznetsov's and Blackstock's equations with Gurtin--Pipkin heat flux law as derived in \cite{kaltenbacher2023limting}, and, to fix ideas, recall relevant classes of memory kernels in the context of acoustics. Section~\ref{Sec:Kuznetsov_prop}, will be dedicated to studying Kuznetsov's equation. We first first show uniform-in-$\eps$ well-posedness in Theorem~\ref{Thm:Wellp_2ndorder_nonlocal_Kuzn}, which will enable us to state the main result of the section relating to the continuity of the solution with respect to the parameter $\eps$; see Theorem~\ref{Thm:Limit}. From there, one can extract the limiting behavior of the equation as the parameter vanishes; see Corollary~\ref{Corollary:Limit_epsK} and Proposition~\ref{Prop:Limit_a<=b}. Section~\ref{Sec:Blackstock_prop} follows a similar organization and is dedicated to extending the results to the Blackstock setting. So as not to burden the presentation, we mainly focus in Section~\ref{Sec:Blackstock_prop} on the differences in the analysis between the Kuznetsov and Blackstock settings.

\subsection*{Notational conventions} Below, we often use the notation $A\lesssim B$ for $A\leq C\, B$ with a constant $C>0$ that may depend on final time $T$ and the spatial domain $\Omega$, but never on the small parameter $\eps$.
\\\indent We assume throughout that
$\Omega \subset \R^n$, where $n \in \{1, 2, 3\}$, is a bounded $C^4$-regular domain and introduce the following functional Sobolev spaces which are of interest in the analysis:
\begin{equation} \label{sobolev_withtraces}
	\begin{aligned}
		\Honetwo:=&\,H_0^1(\Omega)\cap H^2(\Omega), \ 
		\Honethree:=\, \left\{u\in H^3(\Omega)\,:\,  u|_{\partial\Omega} = 0, \  \D u|_{\partial\Omega} = 0\right\},\\ 
		\Honefour:=&\,\left\{u\in H^4(\Omega)\,:\, u|_{\partial\Omega} = 0, \  \D u|_{\partial\Omega} = 0 \right\}.
	\end{aligned}
\end{equation}
\\\indent Given final time $T>0$ and $p, q \in [1, \infty]$, we use $\|\cdot\|_{L^p (L^q(\Om))}$ to denote the norm on $L^p(0,T;L^q(\Omega))$ and $\|\cdot\|_{L^p_t (L^q(\Om))}$ to denote the norm on $L^p(0,t;L^q(\Omega))$ for $t \in (0,T)$. We use $(\cdot, \cdot)_{L^2}$ for the scalar product on $\Ltwo$.\\
\indent We denote by $\Lconv$ the Laplace convolution: 
$(f\Lconv g)(t)=\int_0^t f(t-s)g(s)\ds$. Thus, by defining the Abel kernel:
\begin{equation} \label{def_galpha}
	g_\alpha(t):= \frac{1}{\Gamma(\alpha)} t^{\alpha-1}, \quad \alpha \in (0,1) 
\end{equation}
and introducing the notational convention
\begin{equation} \label{def_g0}
	g_0:= \delta_0,
\end{equation}
where $\delta_0$ is the Dirac delta distribution, we may define the Djrbashian--Caputo fractional derivative, for $w \in W^{1,1}(0,t)$, by
\[
\Dt^{\eta}w(t)=g_{\lceil\eta\rceil - \eta} \Lconv \Dt^{\lceil\eta\rceil} w, \qquad -1<\eta <1,
\]
where $\lceil\eta\rceil$ is the smallest integer greater than or equal to $\eta$; see, for example,~\cite[\S 1]{kubica2020time} and~\cite[\S 2.4.1]{podlubny1998fractional}. When $\eta<0$, $\Dt^\eta$ is interpreted as a fractional integral.
\section{Modeling and relevant classes of kernels} \label{Sec:AcousticModeling}
Nonlinear acoustic equations with fractional dissipation arise as models of sound propagation through media with anomalous diffusion which can be described by a Gurtin--Pipkin flux law~\cite{gurtin1968general}. These laws are nonlocal-in-time relations between the heat flux $\bfq$ and the gradient of the temperature of the medium $\theta$:
\[\bfq= \frakK_{\delta,\tau} \Lconv \nabla \theta,\]
where $\frakK_{\delta,\tau}$ may depend on the thermal conductivity, $\kappa$ (and, \emph{in fine}, the sound diffusivity, $\delta>0$), and the thermal relaxation characteristic time, $\tau$.
We discuss below two classes of kernels that are important in the context of nonlinear acoustics.
\subsubsection*{(I) Abel kernels} The Abel memory kernel is given by
\begin{equation} \label{def_fractional_kernel}
	\frakK_{\delta} = \delta\rt^{-\alpha} g_{\alpha},
\end{equation}
with $g_\alpha$ defined in \eqref{def_galpha}.
The constant $\rt>0$ in \eqref{def_fractional_kernel} serves as a scaling factor to adjust for the dimensional inhomogeneity introduced by the fractional integral, in the way done in \cite[Appendix B.4.1.2]{holm2019waves}. 
\subsubsection*{(II) Mittag-Leffler-type kernels}\label{Sec:Gurtinpipkin_laws} These kernels depend on the thermal relaxation time $\tau <<1$ and are a generalization of the widely studied exponential kernel. They are given by 
\begin{equation} \label{ML_kernels}
	\begin{aligned}
		\frakK_\tau = \delta \left(\frac{\rt}{\tau}\right)^{a-b}\frac{1}{\tau^b}t^{b-1}E_{a,b}\left(-\left(\frac{t}{\tau}\right)^a\right),
	\end{aligned}
\end{equation}
where the scaling $\rt$ ensures dimensional homogeneity. We recall that the generalized Mittag-Leffler function is given by 
\begin{align} \label{def_MittagLeffler}
	E_{a, b}(t) = \sum_{k=0}^\infty \frac{t^k}{\Gamma(a k + b)}, \qquad a > 0,\ t,\, b \in \mathbb{R};
\end{align}
see, e.g.,~\cite[Ch.\ 2]{kubica2020time}.  {The case $a=b=1$ leads to the exponential kernel
	\begin{equation} \label{exponential_kernel}
	\frakK_\tau(t) = \frac{\delta}{\tau} \exp\left(-\frac{t}{\tau}\right).
	\end{equation}\\
	 \indent
As noted by~\cite{povstenko2011fractional}, the Mittag-Leffler functions allow one to recast, into a Gurtin--Pipkin form, the Compte--Metzler heat flux laws~\cite{compte1997generalized}, given by:
\begin{alignat}{3}
	\hspace*{-1.5cm}\text{\small(GFE I)}\hphantom{II}&& \qquad \qquad(1+\tau^\alpha \Dal)\boldsymbol{q}(t) =&&\, -\kappa {\rt^{1-\alpha}}\Doal \nabla \theta;\\[1mm]
	\hspace*{-1.5cm}\text{\small(GFE II)}\, \hphantom{I}&&\qquad \qquad (1+\tau^\alpha \Dal)\boldsymbol{q}(t) =&&\, -\kappa {\rt^{\alpha-1}}\Dt^{\alpha-1} \nabla \theta;\\[1mm]
	\hspace*{-1.5cm}\text{\small(GFE III)}\,\, && \qquad \qquad (1+\tau \partial_t)\boldsymbol{q}(t) =&&\, -\kappa {\rt^{1-\alpha}}\Doal \nabla \theta; \\[1mm]
	\hspace*{-1.5cm}\text{\small (GFE)}\hphantom{III}&&\qquad \qquad  (1+\tau^\alpha \Dal)\boldsymbol{q}(t) =&&-\kappa \nabla \theta,\hphantom{{\rt^{1-\alpha}}\Doal }
\end{alignat}
where $\kappa$ has the usual dimension of thermal conductivity. 
In particular, the parameters $(a,b)$ of the Mittag-Leffler kernel should be chosen as in Table~\ref{tab:ker_par_mod}. Here, $\alpha\in(0,1)$ ($\alpha \in (1/2,1)$ for GFE I~\cite{zhang2014time}) is a fractional differentiation parameter. \\
\begin{center}
	\begin{tabular}[h]{|c||c|c|c|c|}
		\hline
		~& GFE I & GFE II & GFE III & GFE \\
		\hline \hline
		$a$&$\alpha$ & $\alpha$ & $1$ & $\alpha$ \\
		\hline
		$b$	&$2\alpha-1$ & $1$ & $\alpha$ & $\alpha$ \\
		\hline
	\end{tabular}
	~\\[2mm]
	\captionof{table}{Parameters for the Mittag-Leffler kernels motivated by the Compte--Metzler laws}\label{tab:ker_par_mod}
\end{center}
We refer the reader to, e.g, \cite{jordan2014second,kaltenbacher2022time,kaltenbacher2023limting} for details on how the assumed heat flux affects the derivation of acoustic models.
\subsubsection*{Properties of the Mittag-Leffler functions}
When  $1 \geq b \geq a > 0$, the Mittag-Leffler kernel~\eqref{ML_kernels} is completely monotone; see~\cite[Corollary 3.2]{jin2021fractional}. Further, we will rely on the asymptotic behavior of Mittag-Leffler functions when establishing the convergence rates of the studied equation in Section~\ref{Sec:limits_Kuzn}. We recall a useful result to this end:
\begin{align}  \label{asymptotics}
	E_{a,b}(-x)\sim\frac{1}{\Gamma(b-a)\, x} {\mbox{ as }x\to\infty} \qquad \textrm{ where } b\geq a>0;
\end{align}
see, e.g., \cite[Theorem 3.2]{jin2021fractional}. 

\subsection*{Unifying the physical parameters} We wish to investigate the limiting behavior in $\delta$ and in $\tau$ of the resulting nonlocal Kuznetsov and Blackstock equations. Since the uniform well-posedness analysis in both cases is qualitatively the same, we unify the physical parameters $\delta$ and $\tau$ into one parameter $\eps$. Thus, setting 
\begin{equation}
	\frakKeps= \eps \frakK \quad \text{with } \ 
\frakK(t)= \left\{
	 \begin{aligned}
		& 	\rt^{-\alpha} g_{\alpha}(t)\\
		& \textrm {or} \\[-2mm]
		& \left(\frac{\rt}{\tau}\right)^{a-b}\frac{1}{\tau^b}t^{b-1}E_{a,b}\left(-\left(\frac{t}{\tau}\right)^a\right) \ \textrm{with $\tau$ fixed}
	\end{aligned}
	\right.
\end{equation}
will allow us to cover Abel and Mittag-Leffler kernels and study their vanishing sound diffusivity limit,
whereas choosing \begin{equation}
	\begin{aligned}
		\frakKeps
		=&\,  \delta \left(\frac{\rt}{\eps}\right)^{a-b} \frac{1}{\eps}\frakK\left(\frac{t}{\eps}\right) \quad \text{with } \ \frakK(t)= t^{b-1} E_{a,b}(-t^a)
	\end{aligned}
\end{equation}
covers the setting of Mittag-Leffler kernels and let the thermal relaxation tend to zero there.

Using the heat flux kernels introduced above, we can state the quasilinear wave equations of interest, namely,
the nonlocal wave equation of Kuznetsov type~\cite{kuznetsov1971equations}
\begin{equation} \label{Kuznetsov_nonlocal}
	(1+2\tk\psi_t)\psi_{tt}-c^2 \Delta \psi -  \frakKeps* \D\psi_{t}+ \ell\partial_t |\nabla \psi|^2=0
\end{equation}
and the nonlocal wave equation of Blackstock type~\cite[p.\ 20]{blackstock1963approximate} 
\begin{equation} \label{Blackstock_nonlocal}
	\begin{aligned}
		\begin{multlined}[t] \psi_{tt}-c^2 (1-2\tk\psi_t)\Delta \psi - \frakKeps* \D\psi_{t}+ \ell\partial_t |\nabla \psi|^2=0,  \end{multlined}.
	\end{aligned}
\end{equation}
For their derivation, we refer to~\cite[Section 2]{kaltenbacher2023limting}. Above, $\tk$, $\ell$ are real constants. The equations are expressed in terms of the acoustic velocity potential $\psi=\psi(x,t)$, which is related to the acoustic pressure $u$ by
\[
u = \varrho \psi_t,
\] 
where $\varrho$ is the medium density. As discussed in~\cite{povstenko2011fractional}, different choices of the kernel $\frakKeps$ lead to a rich family of flux laws that have appeared in the literature.

\subsection*{Assumptions on the memory kernel for both models}
In the considered equations, \eqref{Kuznetsov_nonlocal} and \eqref{Blackstock_nonlocal}, we refer to the nonlinearity in the leading term as Kuznetsov nonlinearity and in the second term as Blackstock nonlinearity. The type of nonlinearity present in the equation naturally plays a role in the uniform well-posedness (and  thus limiting) analysis. More precisely,  Kuznetsov and Blackstock nonlinearities will require different coercivity assumptions on the kernel; we refer to Sections~\ref{Sec:Kuznetsov_prop} and~\ref{Sec:Blackstock_prop} for details. \\
\indent However, to treat both equations, we need a uniform boundedness assumption on the kernel, which we state next. Note that because we intend to study the limiting behavior of the models as $\eps \searrow 0$, we may restrict our attention in the  analysis  on an interval $(0, \beps)$ for some fixed $\beps>0$ without loss of generality. Let final time $T>0$ and $\beps>0$.
\begin{center}
	\fbox{ 
		\begin{minipage}{0.9\textwidth}
			Given $\eps \in (0, \beps)$, the kernel satisfies
			\begin{equation}\label{assumption1} \tag{\ensuremath{\bf {A}_1}}
				\frakKeps \in L^1(0,T) \cup \{\eps \delta_0\}  
			\end{equation}
			with the $\varepsilon$-uniform bound:
			\begin{equation}
				\|\frakKeps\|_{\calM(0,T)} \leq \cAone.
			\end{equation}
		\end{minipage}
	}
\end{center}
We use $\|\cdot\|_{\calM(0,T)}$ to denote the total variation norm, which in our context should be understood as:
\begin{equation} \label{def_calM}
	\|\frakKeps\|_{\mathcal{M}(0,T)}=\begin{cases}
		\varepsilon &\text{ if }\frakKeps= \varepsilon \delta_0,\\
		\|\frakKeps\|_{L^1(0,T)}&\text{ if }\frakKeps \in L^1(0,T).
	\end{cases}
\end{equation}
Furthermore, in order to extract sufficient regularity for the fixed-point proof we perform a bootstrap argument; see, e.g, estimate~\eqref{eq:bootstrap_Kuzn} below. The following assumption will be needed to control the arising convolution term.
\begin{center}
	\fbox{ 
		\begin{minipage}{0.9\textwidth}
			For all $y\in L^2(0, T; \Ltwo)$ and $t\in(0,T)$, 			\begin{equation}\label{assumption2}  \tag{\ensuremath{\bf{A}_{2}}}
				\int_0^t\bigl((\frakKeps\Lconv y_t)(s), y(s)\bigr)_{L^2(\Omega)}\ds\geq - \cAtwo\, \|y(0)\|^2_{L^2(\Omega)}, \quad
				y\in W^{1,1}(0,t;L^2(\Omega))
			\end{equation}
			where $\cAtwo>0$ does not depend on $\varepsilon \in(0,\beps)$. 
	\end{minipage}}
\end{center}

\section{Analysis of the nonlocal Kuznetsov equation}\label{Sec:Kuznetsov_prop}
In this section, we focus on the Kuznetsov equation~\eqref{Kuznetsov_nonlocal}.
We present the specific kernel assumptions needed in the well-posedness and the limiting analysis pertaining to this equation.  In Section~\ref{Sec:Blackstock_prop}, we point out the differences and subtleties of treating the Blackstock equation.
\subsection*{Kuznetsov-specific assumption on the memory kernel} \label{Sec:Assumption_Kuzn}
To treat the Kuznetsov nonlinearity, we need to be able to extract some regularity from the $\frakKeps$ term after testing. More precisely, we require the following:
\begin{center}
	\fbox{ 
		\begin{minipage}{0.9\textwidth}
			For all $y\in L^2(0, T; \Ltwo)$, it holds that 
			\begin{equation}\label{assumption3} \tag{\ensuremath{\bf {A}^\textbf{K}_{3}}}
			\begin{aligned}
			\int_0^{t} \intO \left(\frakKeps* y \right)(s) \,y(s)\dxs\geq \cAthree
			\int_0^{t} \|(\frakKeps* y)(s)\|^2_{\Ltwo} \ds 
			\end{aligned}
			\end{equation}
		for all $t\in(0,T)$, where $\cAthree>0$ does not depend on $\varepsilon \in(0,\beps)$.
		\end{minipage}}
\end{center}
\subsection{Uniform well-posedness of the Kuznetsov equation with fractional-type dissipation}
\indent We consider the following initial-boundary value problem: 
\begin{equation}\label{ibvp_Kuzn_general}
\left \{	\begin{aligned} 
&(1+2\tk\peps)\pepstt-c^2 \Delta \peps - \Delta \frakKeps * \pepst + 2 \ell \,\nabla \peps \cdot \nabla\pepst=  0\quad  &&\text{in } \Omega \times (0,T), \\
&\peps =0 \quad  &&\text{on } \partial \Omega \times (0,T),\\
&(\peps, \pepst)=(\psi_0, \psi_1), \quad  &&\text{in }  \Omega \times \{0\}.
\end{aligned} \right.
\end{equation}
\\\indent
Our aim, initially, is to establish the well-posedness of~\eqref{ibvp_Kuzn_general}, uniformly in $\eps$. This result will be the basis for the subsequent study of the limiting behavior. We introduce the mapping
$
\TK:\phi \mapsto \peps$,
where $\phi$ will belong to a ball in a suitable Bochner space and $\peps$ will solve the linearized problem 
\begin{subequations} \label{ibvp_2ndorder_lin_Kuzn}
	\begin{equation} \label{ibvp_2ndorder_lin_Kuzn:Eq}
	\begin{aligned}
	\aaa \peps_{tt}-c^2 \Delta \peps - \frakKeps* \D\peps_{t}+ \nabla \lll \ \cdot \nabla \peps_t=0 \ \text{in }\Omega\times(0,T),
	\end{aligned} 
	\end{equation}
	with variable coefficients
	\begin{align}
	\aaa=1+2\tk \phi_t, \qquad \lll= 2 \ell \phi, \qquad \tk, \ell \in \R,
	\end{align}
	supplemented by the initial and boundary conditions:
	\begin{equation} \label{ibvp_2ndorder_lin_Kuzn:Conditions}
	\begin{aligned}
	(\peps, \peps_t) \vert_{t=0}=\,(\psi_0, \psi_1), \qquad \peps \vert_{\partial \Omega}=0,
	\end{aligned} 
	\end{equation}
\end{subequations}
the idea being that a fixed-point of this mapping ($\phi=\peps$) would solve the nonlinear problem.
\subsubsection{ Uniform well-posedness of a linear problem with variable coefficients} \label{Subsection:LinWellp}
\indent The well-definedness of the mapping and the fixed-point argument rest on the uniform well-posedness of the linear problem, which we therefore consider first.
This linear analysis is conducted using an energy method with a smooth semi-discretization in space and weak compactness arguments. In particular, the uniform analysis goes through by testing the time-differentiated PDE with $\Delta^2 \pepstt$ in combination with a bootstrap strategy. 
\begin{proposition}
	\label{Prop:WellP_Lin_2ndorder_nonlocal}
	Let assumptions \eqref{assumption1}, \eqref{assumption2}, and \eqref{assumption3} on the memory kernel hold. Let $\eps \in (0, \beps)$, $\tk$, $\ell \in \R$,  $T>0$, and let
	\begin{equation}
	\phi \in X_\phi:=W^{2,\infty}(0,T;\Honetwo)
	\cap W^{1,\infty}(0,T;\Honethree)
	\cap L^2(0,T;\Honefour).
	\end{equation}
	Assume that there exist $\overline{\aaa}$ and $\underline{\aaa}$, independent of $\eps$, such that the nondegeneracy condition
	\begin{equation} \label{non-deg_phi}
	0<\ulal \leq \aaa(\phi)=1+2\tk \phi_t(x,t)\leq \olal \quad \text{a.e. in } \ \Omega \times (0,T), 
	\end{equation} 
	holds. 
	Furthermore, assume that the initial conditions satisfy
	\begin{equation}
	(\psi_0, \psi_1) \in \begin{cases}
	\Honefour\times \Honethree \ \text{if } \frakKeps = \eps \delta_0 \ \text{or }  \frakKeps \equiv 0, \\[2mm]
	\Honefour\times \Honefour \ \text{if } \frakKeps \not \equiv 0 \in L^1(0,T) . 
	\end{cases}
	\end{equation}
	Then there exists $\Xi>0$, independent of $\eps$, such that if
	\begin{equation}\label{def_xi}
	\begin{aligned}
		\begin{multlined}[t] 
	\|\phi\|_{L^1(H^4(\Omega))} +	\|\phi_{t}\|_{L^1(H^3(\Omega))} +	\|\phi_{tt}\|_{L^1(H^2(\Omega))} \leq \Xi,
		\end{multlined}
	\end{aligned}
	\end{equation}
	and if the final time $T=T(\|\phi\|_{X_\phi})$ is small enough compared to $\cAthree$, then there is a unique  solution $\peps$ of \eqref{ibvp_2ndorder_lin_Kuzn} in
	\begin{equation}\label{solution_space_psi_Kuzn}
		\begin{aligned}
			X_{\psi} =& \,\begin{multlined}[t] W^{3,1}(0,T;\Honezero) \cap W^{2,\infty}(0,T;\Honetwo) \\
				  \cap W^{1,\infty}(0,T;\Honethree) \cap L^2(0,T;\Honefour).
			\end{multlined}
		\end{aligned}
	\end{equation} 
	This solution satisfies the estimate 
	\begin{equation} \label{Lin2ndorderNonlocal_Kuzn:Main_energy_est}
	\begin{aligned}
	& \begin{multlined}[t] \|\peps\|^2_{X_\psi}+ \int_0^{T} \|(\frakKeps* \nabla \D \pepstt)(s)\|^2_{\Ltwo} \ds 
	\end{multlined}\\
	{\leq C_{\textup{lin}}(T)\Bigl(}&\,  \|\psi_0\|^2_{H^4(\Om)}+\|\psi_1\|^2_{\Hthree}+ \eta \|\frakKeps\|_{L^1(0,T)} ^2\|\D^2 \psi_1\|^2_{\Ltwo}
	{\Bigr)},
	\end{aligned}
	\end{equation}
	where $\eta = 0$ if $\frakKeps \in \{ 0, \eps\delta_0 \}$ and $\eta =1$ otherwise. 
The constant $C_{\textup{lin}}(T)$ does not depend on $\eps$.
\end{proposition}

	\subsection*{Discussion of the statement} Before proceeding to the proof, let us discuss the statement made above.
\begin{itemize}
 \item Proposition~\ref{Prop:WellP_Lin_2ndorder_nonlocal} ensures well-posedness of the linearized problem \eqref{ibvp_2ndorder_lin_Kuzn} under the condition that the final time $T=T(\|\phi\|_{X_\phi})$ is small enough compared to $\cAthree$. 
 The high order in space of the testing strategy combined with the time-fractional evolution lead to such a condition.
 \\\indent
 While the relationship between $T$ and $\cAthree$ is made more precise below (see inequality \eqref{est_kernel_terms}), we want to point out that one of the ways to extend the well-posedness final time $T$ is by making $\cAthree$ larger which can be ensured if we have a smaller $\delta$-parameter (see discussion page~\pageref{paragaraph:delta_discussion} for explanations on the relationship between $\delta$ and $\cAthree$).
\item It is noteworthy that we need more regular data when $\frakKeps \not \equiv 0 \in L^1(0,T)$.
	This is primarily due to needing to work with the time-differentiated equation in the analysis (see \eqref{time_diff_discr_2ndorder_nonlocal_Kuzn}) in order to extract sufficient regularity of the solution to treat the nonlinearities. For that we will use the differentiation rule
	$$ \left(\int_0^t \frakKeps(t-s) \D \pepst(s)\ds\right)_t = \int_0^t \frakKeps(t-s) \D \pepstt(s)\ds + \frakKeps(t) \D \peps_{t}(0),$$
	which invokes a regularity assumption on $\D \pepst(0)$.
\end{itemize}
\begin{proof}
We conduct the analysis based on energy arguments performed on a Galerkin semi-discretization in space of the problem; see, for example, \cite{kaltenbacher2023limting} for similar arguments. We emphasize that compared to \cite{kaltenbacher2023limting}, the presence of the quadratic gradient nonlinearity ($2\ell\,\nabla \peps \cdot \nabla \pepst)$ combined with the time-derivative in the quasilinear coefficient ($1+2\tk\pepst$) in equation~\eqref{ibvp_Kuzn_general} will necessitate extracting more regularity from the linearized equation's solution, in turn leading to a higher-order testing.\\
\indent	 The existence of a unique approximate solution follows by reducing the semi-discrete problem to a system of Volterra integral equations of the second kind, employing usual arguments; cf.~\cite{kaltenbacher2022time} and~\cite[Appendix A]{kaltenbacher2023limting}.  We omit them here. 
For notational simplicity, we drop the superscript discretization parameter 
	when denoting the approximate solution.\\
	\indent  We follow the approach of~\cite{kaltenbacher2022parabolic}, where the local-in-time Kuznetsov equation is studied, and consider a time-differentiated semi-discrete problem. We set $p=\peps_t$, where the dependence of $p$ on $\eps$ is omitted to simplify the notation. 
	The time-differentiated semi-discrete equation is then given by
	\begin{equation} \label{time_diff_discr_2ndorder_nonlocal_Kuzn}
	\aaa p_{tt}-c^2\Delta p - \frakKeps* \D p_{t}= F
	\end{equation}
	with the right-hand side
	\begin{align} \label{def_F}
	F(t)= -\nabla \lll_t \cdot \nabla p-\nabla \lll \cdot \nabla p_t 
	-\aaa_t p_t + \eta \frakKeps(t) \D p(0),
	\end{align}
	where $\eta$ is as in the statement of the proposition and $p(0)= \psi_1$. We intend to multiply \eqref{time_diff_discr_2ndorder_nonlocal_Kuzn} with $(-\Delta)^2p_t$ and integrate over space and $(0,t)$. When integrating by parts in the estimates below, we rely on the fact that $p_{tt}=\Delta p=\Delta p_t=0$ on the boundary for sufficiently smooth Galerkin approximations.\\
	\indent We note that
	\begin{equation}
	\begin{aligned}
	&(\aaa p_{tt}, (-\Delta)^2 p_{t})_{L^2} \\
	=&\,  \frac12\ddt(	\aaa\D p_{t}, \Delta p_t)_{L^2}-\frac12(\aaa_t\D p_{t}, \Delta p_t)_{L^2}+(p_{tt}\, \D\aaa+2\nabla p_{tt}\cdot\nabla\aaa, \Delta p_t)_{L^2} .		
	\end{aligned}
	\end{equation}
	Additionally, 
	\begin{equation} \label{c_term_est}
	\begin{aligned}
	&-c^2( \Delta p, (-\Delta)^2 p_{t})_{L^2}
	=&\, \frac{c^2}{2}\ddt  ( \nabla \Delta p, \nabla \Delta p)_{L^2}.
	\end{aligned}
	\end{equation}
	By the time-differentiated semi-discrete PDE  \eqref{time_diff_discr_2ndorder_nonlocal_Kuzn}, we have $F=0$ on the boundary and thus
	\[
	(F, (-\Delta)^2p_t)_{L^2}=(\Delta F, \Delta p_t)_{L^2}.
	\]
	Therefore, testing \eqref{time_diff_discr_2ndorder_nonlocal_Kuzn} with $(-\Delta)^2 p_t$ and using assumption \eqref{assumption3} yields the inequality 
	\begin{equation}\label{ineq_1}
	\begin{aligned}
	&\frac12 \nLtwo{\sqrt{\aaa}\D p_t}^2 \big \vert_0^t
	+\frac{c^2}{2} \nLtwo{\nabla\Delta p}^2\big \vert_0^t 
	+ \cAthree \int_0^{t} \|(\frakKeps* \nabla \D \pt)(s)\|^2_{\Ltwo} \ds \\
	\leq&\, \begin{multlined}[t] \int_0^t\Bigl(
	\frac12(\aaa_t\D p_t,\D p_t)_{L^2}
	-( p_{tt}\, \D\aaa+2\nabla p_{tt}\cdot\nabla\aaa,\D p_t)_{L^2}
	\Bigr)\ds \\ +\int_0^t (\D F,\D p_t)_{L^2} \ds. \end{multlined}
	\end{aligned}
	\end{equation}
We note that the value of $\|\D p_t (0)\|_{L^2(\Om)}=\|\D \psi_{tt} (0)\|_{L^2(\Om)}$, can be bounded by testing \eqref{ibvp_Kuzn_general} at $t=0$ by $\D^2 p_t(0)$ to obtain:
	\[
	\| \D p_t (0) \|_{L^2(\Om)} \leq \left\|\D\left[\frac 1\aaa \left(c^2\D\psi_0 - 2\ell \nabla \psi_0 \cdot \nabla \psi_1 \right)\right]\right\|_{L^2(\Om)}.
	\]
	\\
	The $\aaa$ terms on the right-hand side of \eqref{ineq_1} can be estimated as follows. We have
	\[
	\begin{aligned}
	\int_0^t\frac12(\aaa_t\D p_t,\D p_t)_{L^2}\ds
	\leq\,&\frac12\|\aaa_t\|_{L^1(\Linf)} \|\D p_t\|_{L^\infty_t(\Ltwo)}^2 \\ \leq\,& \|\phi_{tt}\|_{L^1(\Linf)} \|\D p_t\|_{L^\infty_t(\Ltwo)}^2 \lesssim \Xi\|\D p_t\|_{L^\infty_t(\Ltwo)}^2,
	\end{aligned}
	\]
	where in the last inequality, we have used the embedding $H^2(\Omega)\hookrightarrow L^\infty(\Omega)$ combined with \eqref{def_xi}.\\\indent  
	Secondly, by H\"older and  Poincar\'e--Friedrichs as well as Young's inequalities, and the embedding $H^1(\Omega) \hookrightarrow L^6(\Omega)$, we have
	\begin{equation} \label{est_with_nabla_p_tt_}
	\begin{aligned}
	&\int_0^t
	-(p_{tt}\, \D\aaa + \nabla p_{tt}\cdot\nabla\aaa,\D p_t)_{L^2}
	\ds  \\
		\lesssim&\,  \|\aaa\|_{L^\infty(H^3(\Om))}\left(  
		\|\nabla p_{tt}\|_{L^1_t(\Ltwo)}^2+\gamma\|\D p_t\|_{L^\infty_t(\Ltwo)}^2 \right) \\
	\lesssim&\, \|\phi\|_{X_\phi}\left(  
	\|\nabla p_{tt}\|_{L^1_t(\Ltwo)}^2+\gamma\|\D p_t\|_{L^\infty_t(\Ltwo)}^2 \right)
	\end{aligned}
	\end{equation}
for any $\gamma>0$.	The second term on the right-hand side above will be absorbed by the left-hand side for sufficiently small $\gamma$. To estimate  $\|\phi\|_{X_\phi}\|\nabla p_{tt}\|_{L^1_t(\Ltwo)}^2$, we use the (semi-discrete) time-differentiated PDE to express $\nabla \ptt$:
	\begin{equation}
	\begin{aligned}
	\nabla p_{tt}= \nabla \left[ \frac{1}{\aaa} \left( c^2 \bbb\Delta p +\frakKeps* \D p_{t}+ F \right)\right].
	\end{aligned}
	\end{equation}
	From here and the uniform boundedness of $\aaa$ in \eqref{non-deg_phi}, we have 
	\begin{equation}\label{est_nabla_p_tt_}
	\begin{aligned}
	&\|\nabla p_{tt}\|_{L^1_t(\Ltwo)} \\
	\lesssim&\, \begin{multlined}[t]
	\|\nabla\aaa\|_{L^\infty(\Linf)} \left\{\|\Delta p\|_{L^1_t(\Ltwo)} + \left \| \frakKeps *\D p_{t}\right\|_{L^1_t(\Ltwo)}+\|F\|_{L^1_t(\Ltwo)} \right\}\\
	+\|\nabla \Delta p\|_{L^\infty_t(\Ltwo)}  
	+  \left\| \frakKeps* \nabla\D p_{t}\right\|_{L^1(\Ltwo)}+\|\nabla F\|_{L^1_t(\Ltwo)}.\end{multlined}
	\end{aligned}
	\end{equation}
	Recall that in view of \eqref{est_with_nabla_p_tt_}, we need an estimate of $	\|\phi\|_{X_\phi}\|\nabla p_{tt}\|_{L^1(\Ltwo)}^2$. After squaring \eqref{est_nabla_p_tt_}, the corresponding kernel terms  can be absorbed by the left-hand side of \eqref{ineq_1} by relying on the smallness of $T\|\phi\|_{X_\phi}$ in front of them (relative to $\cAthree$):
	\begin{equation} \label{est_kernel_terms}
	\begin{aligned}
	&\|\phi\|_{X_\phi}
	\left({\|\nabla\aaa\|^2_{L^\infty(\Linf)}}\left\| \frakKeps* \D p_{t}\right\|_{L^1_t(\Ltwo)}^2
	+\left\|\frakKeps *\nabla\D p_{t} \right\|_{L^1_t(\Ltwo)}^2\right)\\
	\lesssim&\,  (1+\|\phi\|^2_{X_\phi})\|\phi\|_{X_\phi}\left\|\frakKeps *\nabla\D p_{t} \right\|_{L^1_t(\Ltwo)}^2 \\
	\lesssim&\,  (1+\|\phi\|^2_{X_\phi})\|\phi\|_{X_\phi}T\left\|\frakKeps *\nabla\D p_{t} \right\|_{L^2_t(\Ltwo)}^2. 
	\end{aligned}
	\end{equation}
	It remains to estimate the $F$ and $\nabla F$ terms in \eqref{est_nabla_p_tt_}, as well as $\Delta F$ term in \eqref{ineq_1}. Since $F \vert_{\partial \Omega}=0$, we have
	\[
	\| F\|_{L^1_t(H^2(\Om))} \lesssim \|\D F\|_{L^1_t(\Ltwo)}
	\]
	and it is thus sufficient to bound $\D F$. Recalling how $F$ is defined in \eqref{def_F},
	we have 
	\begin{equation} \label{Delta_tildef}
	\begin{aligned}
	\D F
	=\begin{multlined}[t]
	-\D\aaa_t p_t-2\nabla\aaa_t\cdot\nabla p_t-\aaa_t \D p_t
	+\Delta(-\nabla \lll_t \cdot \nabla p-\nabla \lll \cdot \nabla p_{t} )+ \eta \frakKeps(s) \D^2 p(0), \end{multlined}
	\end{aligned}
	\end{equation}
	with further
	\begin{equation} \label{Delta_tildef_}
	\begin{aligned}
	&\Delta(-\nabla \lll_t \cdot \nabla p-\nabla \lll \cdot \nabla p_{t} )\\
	=&\,\begin{multlined}[t] -\nabla \D \lll_t \cdot \nabla p-2D^2 \lll_t : D^2 p-\nabla\lll_t \cdot \nabla \D p\\ -\nabla \D \lll \cdot \nabla p_t-2D^2 \lll: D^2 p_t-{\nabla\lll\cdot}\nabla \D p_t,  \end{multlined}\\
	\end{aligned}
	\end{equation}
	where $D^2 v=(\partial_{x_i} \partial_{x_j} v)_{ij}$ denotes the Hessian. 
	Since the last term in \eqref{Delta_tildef_} is problematic to estimate in the $L^1(0,t; L^2(\Omega))$ norm, given the regularity we can expect from $\nabla \Delta p_t$ based on the left-hand side of \eqref{ineq_1}, we split it in the following manner:
	\begin{equation} \label{Delta_tildef_Deltapt_second}
	\begin{aligned}
	\int_0^t (\D F,\D p_t)_{L^2} \ds
	\leq&\,  
	\|\D F-g\|_{L^1(\Ltwo)} \|\D p_t\|_{L^{\infty}_t(\Ltwo)}
	+\left|\int_0^t (g,\D p_t)_{\Ltwo} \ds\right|,
	\end{aligned}
	\end{equation}
	where $g=\nabla\lll\cdot\nabla \D p_t$. Then, since $\Delta p_t=0$ on $\partial \Omega$, 
	\begin{equation} \label{ineq:g_estimate}
	\begin{aligned}
	\left|\int_0^t (g,\D p_t)_{L^2} \ds\right|
	=\frac12\left|\int_0^t (\D\lll,(\D p_t)^2)_{L^2} \ds\right| 
	\lesssim&\, \|\Delta \lll\|_{L^1(\Linf)} \|\D p_t\|_{L^\infty_t(\Ltwo)}^2\\
	\lesssim& \, 	\Xi \|\D p_t\|_{L^\infty_t(\Ltwo)}^2
	\end{aligned}
	\end{equation}
with $\Xi$ as in \eqref{def_xi}.
	We then estimate the remaining $F-g$ terms as follows:
	\begin{equation}  \label{est_Delta_tildef}
	\begin{aligned}
	&\|\D F-g\|_{L^1(\Ltwo)}\\
	\lesssim&\, \begin{multlined}[t]
	\|\D\aaa_t\|_{L^1(\Ltwo)} \|p_t\|_{L^\infty_t(\Linf)}
	+\|\nabla\aaa_t\|_{L^1(L^3(\Om))} \|\nabla p_t\|_{L^\infty_t(L^6(\Om))}
	\\+\|\aaa_t\|_{L^1(L^\infty(\Om))} \|\D p_t\|_{L^\infty_t(\Ltwo)}
	\\	+\|\nabla \D \lll_t \cdot \nabla p \|_{L^1(\Ltwo)}+\|\nabla \lll_t \cdot \nabla \D p \|_{L^1_t(\Ltwo)}+\|D^2 \lll_t : D^2 p \|_{L^1_t(\Ltwo)}\\
	+\|\nabla \D \lll \cdot \nabla p_t \|_{L^1_t(\Ltwo)}
	+\|D^2 \lll: D^2 p_t \|_{L^1_t(\Ltwo)}
	+ \eta \|\frakKeps\|_{L^1(0,t)} \|\D^2 \psi_1\|_{\Ltwo}, \end{multlined}
	\end{aligned}
	\end{equation}
	where we can further bound the $\lll$ terms above: 
	\begin{equation} \label{estlll1}
	\begin{aligned}
	&\|\nabla \D \lll_t \cdot \nabla p \|_{L^1_t(\Ltwo)}
	{+\|\nabla \lll_t \cdot \nabla \D p \|_{L^1_t(\Ltwo)}}
	+\|\nabla \D \lll \cdot \nabla p_t \|_{L^1_t(\Ltwo)}\\
	\lesssim&\, 	\Xi(\|\nabla p \|_{L^\infty_t(H^2(\Om))}+\|\D p_t \|_{L^\infty_t(\Ltwo)}).
	\end{aligned}
	\end{equation}
	\noindent By elliptic regularity, the Hessian satisfies 
	\[
	\|D^2 v\|_{L^p(\Om)}\leq C_{\textup{H}} \|\Delta v\|_{L^p(\Om)} 
	\]
	for all $v\in H_0^1(\Omega)\cap W^{2,p}(\Omega)$ where $p\in (1,6]$; see, e.g., \cite[Theorem 2.4.2.5]{grisvard2011elliptic}. Therefore, together with the embedding $H^1(\Omega) \hookrightarrow L^6(\Omega)$ and Poincar\'e--Friedrichs inequality, we have the bound
	\begin{equation} \label{estlll2}
	\begin{aligned}
	\|D^2 \lll_t : D^2 p \|_{L^1(\Ltwo)} \lesssim \|D^2 \lll_t \|_{L^1(L^3(\Om))}\|D^2 p \|_{L^\infty_t(L^6(\Om))} 
	\lesssim \, 	\Xi \|\nabla \D p \|_{L^\infty_t(\Ltwo)}
	\end{aligned}
	\end{equation}
	and, similarly, 
	\begin{equation}\label{estlll3_Honefour}
	\|D^2 \lll: D^2 p_t \|_{L^1(\Ltwo)}  \lesssim \|D^2 \lll \|_{L^1(\Linf)} \|\D p_t \|_{L^\infty_t(\Ltwo)} \lesssim 	\Xi \|\D p_t \|_{L^\infty_t(\Ltwo)}.
	\end{equation}
	Altogether, we can write
	\begin{equation}  \label{est_Delta_tildef_second}
	\begin{aligned}
	\|\D F -g \|_{L^1_t(\Ltwo)}
	\lesssim&\, \begin{multlined}[t]
		\Xi\left(\|\D p_t\|_{L^\infty_t(\Ltwo)}+\|\nabla \D p \|_{L^\infty_t(L^2(\Om))} \right)
	\\
	+ \eta \|\frakKeps\|_{L^1(0,t)} \|\D^2 \psi_1\|_{\Ltwo}. \end{multlined}
	\end{aligned}
	\end{equation}
	Thus, combining the derived bounds for sufficiently small $\Xi$ allows us to absorb all right-hand side terms by the left-hand side in \eqref{ineq_1} and thus, combined with \eqref{est_nabla_p_tt_},
	leads to
	\begin{equation}\label{ineq:intermediate}
		\begin{aligned}
			& \|\nabla p_{tt}\|_{L^1_t(\Ltwo)}^2 + \nLtwo{\D p_t}^2			+\nLtwo{\nabla\Delta p}^2
			+  \int_0^{t} \|(\frakKeps* \nabla \D \pt)(s)\|^2_{\Ltwo} \ds \\
			\leq&\,  { C_{\textup{lin}}(T)\Bigl(}  \|\psi_0\|^2_{H^4(\Om)}+ \| \psi_1\|^2_{H^3(\Om)}+ \eta \|\frakKeps\|_{L^1(0,T)} ^2\|\D^2 \psi_1\|^2_{\Ltwo}
			{\Bigr)}.
		\end{aligned}
	\end{equation}
	
	\noindent \emph{Bootstrap argument for $\peps \in L^2(0,T;\Honefour)$}.\label{paragraph:bootstrap} If $\ell \not=0$, then, due to inequality~\eqref{estlll3_Honefour} and the perspective of the fixed-point argument, we have to estimate the term $\D^2 \peps$. To this end, we can employ the additional regularity to be achieved by estimating 
	${\|\Delta^2 \peps\|_{L^2(L^2(\Om))}}$ via a bootstrap argument
	analogously to~\cite[p.\ 22]{kaltenbacher2022parabolic}. 
	Multiplying the identity 
	\begin{align} \label{bootstrap_identity}
		[\frakKeps\Lconv\partial_t+ c^2]\D\peps =r:=\aaa p_t+\nabla \lll \cdot \nabla p
	\end{align}
	with 
	$\D^3\peps$ 
	and integrating by parts, due to assumption \eqref{assumption2}, we have
	\begin{equation}\label{eq:bootstrap_Kuzn}
	\begin{aligned}
		&	c^2\|\D^2\peps\|_{L^2_t(L^2(\Om))}^2\\\leq&\,
		\int_0^t(\D r,\D^2\peps)_{L^2}\ds + \cAtwo \|\Delta^2 \psi_0\|^2_{\Ltwo}\\ 
		\leq& \|\D^2\peps\|_{L^2_t(L^2(\Om))}
		\left\|\D [\aaa p_t-\nabla \lll \cdot \nabla p]\right\|_{L^2_t(L^2(\Om))}  + \cAtwo \|\psi_0\|^2_{H^4(\Om)},
	\end{aligned}
	\end{equation}
	where we have used the fact that $r$ and $\D^2 \peps$ vanish on the boundary.
	Thanks to \eqref{ineq:intermediate} and the fact that $\phi \in X_\phi$, we readily obtain that $\D [\aaa p_t-\nabla \lll \cdot \nabla p] \in L^2(0,T;L^2(\Om)).$ Thus, conbining \eqref{ineq:intermediate} and \eqref{eq:bootstrap_Kuzn}, leads to \eqref{Lin2ndorderNonlocal_Kuzn:Main_energy_est}, at first for the semi-discrete solution. The bound transfers to the continuous setting through standard weak compactness arguments, analogously to \cite[Proposition 3.1]{kaltenbacher2022parabolic} so we omit these details here. 
\end{proof}
\indent For the nonlinear analysis, we intend to use Agmon's inequality with the goal of imposing smallness of data in a less restrictive space; see \eqref{est_Agmon_p_Kuzn} below. Agmon inequality is given by 
\begin{equation}
	\begin{aligned}
		\|u\|_{L^\infty(\Om)} \leq C_{\textup{A}}\|u\|_{L^2(\Om)}^{1-n/4}\|u\|^{n/4}_{H^2(\Om)}
	\end{aligned}
\end{equation} 
for some arbitrary $u\in H^2(\Om)$ and $n \leq 3$; see~\cite[Lemma 13.2]{agmon2010lectures}. 
It is thus helpful to also have a uniform lower-order bound for the solution of \eqref{ibvp_2ndorder_lin_Kuzn}. Under assumptions of Proposition~\ref{Prop:WellP_Lin_2ndorder_nonlocal}, testing \eqref{ibvp_2ndorder_lin_Kuzn} with $\pepst$ and integrating over space and time yields, after usual manipulations, 
\begin{equation} \label{first_est_selfm_Kuzn}
	\begin{aligned}
		&\frac12 \left\{\|\sqrt{\aaa(t)}\pepst(s)\|^2_{L^2(\Om)}+c^2\| \nabla \peps(s)\|^2_{L^2(\Om)} \right\}\Big \vert_0^t+\cAthree \int_0^{t} \|(\frakKeps* \nabla \pepst)(s)\|^2_{\Ltwo} \ds\\
		\leq &\,\begin{multlined}[t] \frac12\|\aaa_t\|_{L^1(L^\infty(\Om))}\|\pepst\|^2_{L^\infty_t(L^2(\Om))}- \int_0^t( \nabla \lll \cdot \nabla \pepst, \pepst)\ds.   \end{multlined}
	\end{aligned}
\end{equation}
To bound the last term on the right, we integrate by parts in space and use that $\pepst\vert_{\partial \Omega}=0$: 
\begin{equation} \label{ineq:lower_order_1}
	\begin{aligned}
		\left|\int_0^t \int_{\Omega} (\nabla \lll  \cdot \nabla \pepst) \pepst \dxs \right| 
		=&\, \frac12 \left | \int_0^t \int_{\Omega} (- \D {\lll})  \pepst \pepst \dxs \right|\\
		\lesssim&\,  \|\D {\lll}\|_{L^1(L^\infty(\Om))}\|\pepst\|_{L^\infty_t(L^2(\Om))}^2.
	\end{aligned}
\end{equation}
Thus
\begin{equation}\label{uniform_lower_order_Kuzn}
	\begin{aligned}
		\frac12 \left\{\|\sqrt{\aaa(t)}\pepst(s)\|^2_{L^2(\Om)}+c^2\| \nabla \peps(s)\|^2_{L^2(\Om)} \right\}\Big \vert_0^t
		\leq  C(\Om)\,\Xi\|\pepst\|^2_{L_t^\infty(L^2(\Om))},
	\end{aligned}
\end{equation} 
where we recall that $\Xi$ is defined in \eqref{def_xi}. Utilizing the above estimates and the smallness of $\Xi$ according to condition \eqref{def_xi}, leads to the sought-after uniform-in-$\eps$ lower order estimate:
\begin{equation}\label{estimate:lower_order_Kuzn}
	 \|\pepst(t)\|^2_{L^2(\Om)} + \| \nabla \peps(t)\|^2_{L^2(\Om)} 
\lesssim \|\psi_1\|^2_{L^2(\Om)} + \| \nabla \psi_0\|^2_{L^2(\Om)}.
\end{equation}
We are now ready to state the uniform well-posedness result for the initial-boundary value problem~\eqref{ibvp_Kuzn_general} for the nonlinear Kuznetsov equation.
\begin{theorem}\label{Thm:Wellp_2ndorder_nonlocal_Kuzn}
		Let $\eps \in (0, \beps)$ and $k$, $\ell \in \R$. Furthermore, let $(\psi_0, \psi_1) \in \Honefour \times \Honefour$ be such that
	\begin{equation}
		\|\psi_0\|^2_{H^4(\Om)}+\|\psi_1\|^2_{H^4(\Om)} \leq r^2,
	\end{equation}
	where $r$ does not depend on $\eps$. 
	Let assumptions \eqref{assumption1}, \eqref{assumption2}, and \eqref{assumption3} on the kernel $\frakKeps$ hold.
	Then, there exist a 
	data size $r_0=r_0(r)>0$ and final time $T=T(r)>0$, both independent of $\eps$, such that if 
	\begin{equation}\label{initsmallWellp_2ndorder_nonlocal_Kuzn}
		\|\psi_0\|^2_{H^1(\Om)}+\|\psi_1\|^2_{L^2(\Om)} \leq r_0^2,
	\end{equation}
	then there is a unique solution $\peps$ in $X_\psi$ (where $X_\psi$ is defined in \eqref{solution_space_psi_Kuzn}) of \eqref{ibvp_Kuzn_general}, which satisfies the following estimate: 
	\begin{equation} \label{Nonlin:Main_energy_est_Kuzn}
		\begin{aligned}
			& \begin{multlined}[t] \|\peps\|^2_{X_\psi}+ \int_0^{T} \|(\frakKeps* \nabla \D \pepstt)(s)\|^2_{\Ltwo} \ds 
			\end{multlined}\\
			\leq&\, C_{\textup{nonlin}}\left(\,  \|\psi_0\|^2_{H^4(\Om)}+\|\psi_1\|^2_{H^3(\Om)}+\eta \,\|\frakKeps\|_{L^1(0,T)} ^2\|\D^2 \psi_1\|^2_{L^2(\Om)}\right),
		\end{aligned}
	\end{equation} 
	where $\eta = 0$ if $\frakKeps \in \{ 0, \eps\delta_0 \}$ and $\eta =1$ otherwise.
	Here, $C_{\textup{nonlin}}=C_{\textup{nonlin}}(\Omega,\ophi, T)$ 
	does not depend on the parameter $\eps$. 
\end{theorem}
\begin{proof}
	The proof follows using the Banach fixed-point theorem in the general spirit of~\cite[Theorem 6.1]{kaltenbacher2022parabolic} which considers $\eps$-uniform local-in-time analysis of the Kuznetsov equation. A similar idea was also used in the study of the uniform-in-$\eps$ Westervelt equation in pressure form~\cite[Theorem 3.1]{kaltenbacher2023limting}. 
	The mapping $\TK:\phi \mapsto \psi$ is introduced on 
	\begin{equation}\label{defB_Kuzn}
		\begin{aligned}
			\mathcal{B} = \begin{multlined}[t]\left\{ \vphantom{{4|\tk|}\|\phi_t\|_{L^\infty(L^\infty(\Om))}}\right. \phi \in X_\psi \, : \, (\phi,\phi_{t})|_{t=0}= (\psi_0, \psi_1), \\
				\ \|\phi\|_{L^1(H^4(\Omega))} +	\|\phi_{t}\|_{L^1(H^3(\Omega))} +	\|\phi_{tt}\|_{L^1(H^2(\Omega))} \leq \Xi, 
				\\
				\left.  
				{4|\tk|}\|\phi_t\|_{L^\infty(L^\infty(\Om))} \leq 1, \ \|\phi\|_{X_{\psi}} \leq R  \vphantom{X_\psi} \ \right\}.\end{multlined}
		\end{aligned}
	\end{equation}
	The radius $R>0$ large enough as well as $\Xi>0$ small enough will be determined below. 
	Note that the set $\mathcal{B} $ is non-empty as the solution of the linear problem with $\tk=\ell=0$ belongs to it if $R$ is sufficiently large.  The smoothness assumptions on $\aaa$, and $\lll$ follow by  $\phi$ being in $\calB$.
	Furthermore, since $ 4 |\tk|\|\phi_t\|_{L^\infty(L^\infty(\Om))} \leq 1$,  $\aaa$ does not degenerate:
	\[
	\begin{aligned}
		\frac12=	\underline{\aaa} \leq \aaa \leq \overline{\aaa}=\frac32.
	\end{aligned}
	\]
	Moreover $\phi\in X_\psi$ implies that $\phi\in X_\phi$, and the smallness condition on $T=T(\|\phi\|_{X_\phi})$ follows for $\phi \in \calB$ if we choose the final time $T$ relative to $R$ and $\cAthree$ according to the assumptions of the linear well-posedness theory given in Proposition~\ref{Prop:WellP_Lin_2ndorder_nonlocal}. \vspace*{2mm}
	
	\emph{$\bullet$ The self-mapping property.} Using the linear well-posedness theory from Proposition~\ref{Prop:WellP_Lin_2ndorder_nonlocal}, we ensure that the solution of the linear problem satisfies 
	\begin{equation} 
		\begin{aligned}
			\|\peps\|^2_{X_\psi}  \leq\,  C_{\textup{lin}}(T)\left( \|\psi_0\|^2_{H^4(\Om)}+\|\psi_1\|^2_{H^3(\Om)}+ \eta \|\frakKeps\|_{L^1(0,T)} ^2\|\D^2 \psi_1\|^2_{L^2(\Om)}
			\right).
		\end{aligned}
	\end{equation}
	 Provided that $R$ is chosen large enough so that 
	\[C_{\textup{lin}}(T) \left(\|\psi_0\|^2_{H^4(\Om)}+\|\psi_1\|^2_{H^3(\Om)}+ \eta \cAone^2\|\D^2 \psi_1\|^2_{L^2(\Om)} \right)\leq R^2,\] the solution $\psi$ satisfies the $R$ bound in \eqref{defB_Kuzn}. 
	Next, we have
	\[
	 \|\phi\|_{L^1(H^4(\Omega))} +	\|\phi_{t}\|_{L^1(H^3(\Omega))} +	\|\phi_{tt}\|_{L^1(H^2(\Omega))} \leq (T+\sqrt T) R, 
	\]
	thus, the condition on $ \|\phi\|_{L^1(H^4(\Omega))} +	\|\phi_{t}\|_{L^1(H^3(\Omega))} +	\|\phi_{tt}\|_{L^1(H^2(\Omega))} \leq \Xi$ will be fulfilled  if we choose the final time $T>0$ small enough so that  \[(T+\sqrt T) R \leq \Xi.
	\]\\
	\indent To finish proving the self-mapping property, it remains to show that the bound $${4|\tk|}\|\psi_t\|_{L^\infty(L^\infty(\Om))} \leq 1$$ in \eqref{defB_Kuzn} holds. Using Agmon's interpolation inequality, we write
	\begin{equation}\label{unif_bound_energynorm_Kuzn}
		\begin{aligned}
			\|\pepst(t)\|_{L^\infty(\Om)} \leq C_{\textup{A}}\|\pepst(t)\|_{L^2(\Om)}^{1-n/4}\|\pepst(t)\|^{n/4}_{H^2(\Om)}.
		\end{aligned}
	\end{equation}
		We combine this with a uniform lower-order estimate in $\eps$ for the linearized problem established in \eqref{uniform_lower_order_Kuzn}: 
	\begin{equation} \label{energy_bound_Kuzn}
		\begin{aligned}
			\|\peps\|^2_{\textup{E}}	\lesssim\|\psi_0\|^2_{H^1(\Om)}+\|\psi_1\|^2_{L^2(\Om)}.
		\end{aligned}
	\end{equation}
	 This approach yields
	\begin{align} \label{est_Agmon_p_Kuzn}
		\|\pepst\|_{L^\infty(L^\infty(\Om))} \leq C_{\textup{A}}\left\{C_\textup{lower}(\Xi)(\|\psi_0\|^2_{H^1(\Om)}+\|\psi_1\|^2_{L^2(\Om)})\right\}^{1/2-{n}/8}R^{{n}/4},
	\end{align}
	where $C_\textup{lower}(\Xi)$ is the hidden constant in \eqref{estimate:lower_order_Kuzn}.
	We choose $r_0>0$ in \eqref{initsmallWellp_2ndorder_nonlocal_Kuzn} small enough, so that
	\[
	4|\tk| C_{\textup{A}}C_\textup{lower}(\Xi)^{1/2-{n}/8} r_0^{1-{n}/4}\,R^{{n}/4} \leq 1,
	\]
	which via \eqref{est_Agmon_p_Kuzn} implies $4|\tk| \|\pepst\|_{L^\infty(L^\infty(\Om))} \leq 1$. Altogether, we conclude that $\TK(\calB) \subset \calB$.\\
		
	\emph{$\bullet$ Contractivity.} We next prove that $\mathcal{T}$ is strictly contractive with respect to the energy norm \eqref{energy_norm}. Take $\phi^{(1)}$ and $\phi^{(2)}$ in $\mathcal{B}$, and denote their difference by $\overline{\phi}= \phi^{(1)} -\phi^{(2)} $. Let $\psi^{\eps,(1)}=\TK (\phi^{(1)})$ and $\psi^{\eps,(2)}=\TK (\phi^{(2)})$. Their difference $\opsi^\eps=\psi^{\eps,(1)} -\psi^{\eps,(2)} \in \mathcal{B}$ then solves 
	\begin{equation} \label{2ndorder_nonlocal_diff_contr_Kuzn}
		\begin{aligned}
			&(1+ 2 \tk \phi_t^{(1)})\opsi^\eps_{tt}-c^2 \Delta \opsi^\eps-  \frakKeps* \D\opsi^\eps_{t}+2 \ell \nabla \phi^{(1)} \cdot \nabla \opsi^\eps_t\\
			=&\, \begin{multlined}[t] -2 \tk \ophi_t \psi^{\eps,(2)} _{tt}-  2\ell \nabla \ophi \cdot \nabla \psi^{\eps,(2)}_t, \end{multlined}
		\end{aligned}
	\end{equation}
	with zero initial and boundary data. We will next test this equation with $\opsi_t^\eps$ and integrate over space and time. As we are in a similar setting to when deriving the lower-order bound in  \eqref{uniform_lower_order_Kuzn} with now
	$
	\aaa=1+2 \tk \phi^{(1)}_t$, 
	and a non-zero source term, we only discuss how to estimate the terms resulting from the  right-hand side contributions.
	The resulting $\tk$ term on the right can be estimated as follows: 
	\begin{equation}
		\begin{aligned}
			\left |-2 \tk \int_0^t \intO \ophi_t \psi_{tt}^{\eps,(2)}  \opsi^\eps_t \dxs \right| 
			\lesssim\,  \|\psi _{tt}^{\eps,(2)}\|_{L^1(L^\infty(\Om))}(\|\opsi^\eps_t\|_{L^\infty_t(L^2(\Om))}^2+\|\ophi_t\|_{L^\infty_t(L^2(\Om))}^2). 
		\end{aligned}
	\end{equation} 
	We can estimate the $\ell$ term similarly:
	\begin{equation}
		\begin{aligned}
			\left|-2\ell\int_0^t \int_{\Omega} \nabla \ophi \cdot \nabla \psi^{\eps,(2)}_t \, \opsi_t^\eps \dxs \right| \lesssim \|\nabla \psi^{\eps,(2)}_t\|_{L^1(L^\infty(\Om))} (\|\opsi_t^\eps\|^2_{L^\infty_t(L^2(\Om))}+\|\nabla \ophi\|_{L^\infty_t(L^2(\Om))}^2).
		\end{aligned}
	\end{equation}
	Thus using that $\|\psi^{\eps,(i)}\|_{X_{\psi}}$, $\|\phi^{\eps,(i)}\|_{X_{\psi}} \leq R$ (with $i=1,2$), we obtain
	\begin{equation}
		\begin{aligned}
			\|\opsi^\eps\|^2_{\textup{E}} \leq 
			C(\Omega) {{T}}R \left(\|\opsi^\eps\|^2_{\textup{E}} + \|\ophi\|^2_{\textup{E}} \right ).
		\end{aligned}
	\end{equation}
	Therefore by reducing ${{T}}$ if needed (independently of $\eps$), 
	we can absorb the energy term $\|\opsi^\eps\|^2_{\textup{E}}$ on the right-hand-side of the inequality and ensure that the mapping $\mathcal{T}$ is strictly contractive with respect to the energy norm \eqref{energy_norm}. The arguments showing that $\calB$ is closed with respect to this norm are analogous to those of \cite[Theorem 4.1]{kaltenbacher2022parabolic}. By the Banach fixed-point theorem, we therefore obtain a unique fixed-point $\peps=\mathcal{T}(\peps)$ in $\calB$, which solves the nonlinear problem. 
\end{proof}
\subsection*{Verifying the assumptions for relevant classes of kernels} 
In Section~\ref{Sec:AcousticModeling}, we saw that Abel and Mittag-Leffler kernels are of particular interest in the context of nonlinear acoustic modeling. We next discuss to which extent they satisfy the assumptions imposed in the preceding analysis. The verification of assumptions \eqref{assumption1} and \eqref{assumption3} on these kernels (with the choice $\eps =\tau$, $\delta>0$ fixed for the Mittag-Leffler kernels) follows from \cite[Section 5]{kaltenbacher2023limting}; we compile them for convenience of the reader in Table~\ref{tab:ker_assu_veri}. \\
\indent The verification in~\cite[Section 5]{kaltenbacher2023limting} of \eqref{assumption3} uses two different techniques: Fourier analysis and the theory of Volterra integrodifferential equations for completely monotone kernels. As per the discussion of \cite[Section 5.2]{kaltenbacher2023limting}, assumption \eqref{assumption3} is verified for the Abel kernel~\eqref{def_fractional_kernel}, and by fixing $\delta = 1$ for the Mittag-Leffler kernels~\eqref{ML_kernels} when $a\leq b$ (which covers the exponential, GFE II, and GFE kernels). When $a>b$ (e.g., for kernels GFE I and GFE III), we need the ratio $\dfrac{\rt}{\eps}$ to be a constant.
\subsubsection*{Influence of the sound diffusivity $\delta>0$ fixed ($\eps =\tau$) on the coercivity constant $\cAthree$}\label{paragaraph:delta_discussion}
As mentioned, the verification of \eqref{assumption3} in \cite[Section 5.2]{kaltenbacher2023limting} for the Mittag-Leffler kernels was performed by fixing $\delta = 1$. Let $y\in L^2(0, T; \Ltwo)$ and $t\in(0,T)$. Then, by \cite[Section 5.2]{kaltenbacher2023limting}, we have for those kernels
			 \begin{equation}
				\begin{aligned}
					\int_0^{t} \intO \left(\frac1\delta\frakKeps* y \right)(s) \,y(s)\dxs\geq \check{c}
					\int_0^{t} \|(\frac1\delta\frakKeps* y)(s)\|^2_{\Ltwo} \ds, 
				\end{aligned}
			\end{equation}
			where $\check{c}>0$ depends neither on $\varepsilon \in(0,\beps)$ nor on $\delta$. Thus, by straightforward manipulations, we can prove that \eqref{assumption3} is verified for $\frakKeps$ with $\cAthree=\dfrac{\check{c}}\delta$, and smaller values of $\delta$ will thus lead to larger values of $\cAthree$.

\indent It remains to discuss the verification of~\eqref{assumption2}.
Assumption \eqref{assumption2} can be checked by means of~\cite[Lemma 5.1]{kaltenbacher2023limting}. Given a Hilbert space $H$, a function  $w \in W^{1,1}(0,T; H)$, and a memory kernel $\frakK \in L^1(0,T)$ verifying that for all $t_0>0$, $\frakK\in W^{1,1}(t_0,T)$, $\frakK\geq0$ on $(0,T)$, and $\frakK'\vert_{[t_0,T]}\leq0$ a.e., ~\cite[Lemma 5.1]{kaltenbacher2023limting} ensures that
\begin{equation}\label{eq:assumption_eq}
	\begin{aligned}
		\int_0^T ( {w}(t), \frakK* {w}_t(t))_{H}\dt \geq&\, 
		\frac12(\frakK* \|{w}\|_H^2)(T)-\frac12\int_0^T\frakK(t)\dt\, \|{w}(0)\|_H^2 
		\\\geq&\, -\frac12 \|\frakK\|_{L^1(0,T)} \|{w}(0)\|_H^2.
	\end{aligned}
\end{equation}
\\ \indent Thus, the kernel assumptions in \cite[Lemma 5.1]{kaltenbacher2023limting} are satisfied up to an $\eps$-dependent final time for the Mittag-Leffler kernels when $a>b$ (with the choice $\eps =\tau$, $\delta>0$ fixed) in the sense that
\begin{equation} \label{A2_tau_time}
	\int_0^t\bigl((\frakKeps\Lconv y_t)(s), y(s)\bigr)_{L^2(\Omega)}\ds\geq - \frac12 \|\frakKeps\|_{L^1(0,t)} \|y(0)\|^2_{L^2(\Omega)}, \quad
	y\in W^{1,1}(0,t;L^2(\Omega))
\end{equation} holds for all \[0 \leq t \leq t_f = t_f(\eps).\] The $\eps$-dependence in the case $a>b$ is owed to the fact that positivity and decreasing monotonicity can only be verified up to such a final time. This dependency is marked by the symbol $\sometimes$ in Table~\ref{tab:ker_assu_veri}.
\\
\indent As for the Abel and Mittag-Leffler kernel with $a\leq b$, they are completely monotone and thus verify the assumptions of \eqref{eq:assumption_eq} for all final times $T \in(0,\infty)$. The constant $\cAtwo$ in \eqref{assumption2} can then be estimated by $\frac12\|\frakKeps\|_{L^1(0,\infty)}$ which is indeed independent of $\eps$; see~\cite[Section 5.1]{kaltenbacher2023limting}. 

\subsubsection*{Influence of a variable sound diffusivity $\eps = \delta \in (0,\beps)$ (with $\tau>0$ fixed) in the Mittag-Leffler kernels}
For the Mittag-Leffler kernels~\eqref{ML_kernels} (including the exponential kernel \eqref{exponential_kernel}), we had a choice between setting $\eps = \delta$ (with $\tau$ fixed) or $\eps=\tau$ (with $\delta$ fixed). Note that the case of $\eps = \delta$ ($\tau$ fixed) is somewhat simpler to discuss as the small parameter $\eps$ appears as a multiplication factor in the kernels. Thus, considering a neighbourhod $(0,\beps)$, it is clear that we can use the upper bound $\beps$ to bound the assumption constants uniformly in $\eps$. \\\indent
To showcase the statement above, consider a generic kernel $\frakK$ (independent of $\eps$) verifying assumptions \eqref{assumption1}, \eqref{assumption2}, and \eqref{assumption3} with constants $\check{c}_1$, $\check{c}_2$, and $\check{c}_3$,
respectively. Then, straightforward manipulations show that $\frakKeps = \eps \frakK$ also verifies \eqref{assumption1}, \eqref{assumption2}, and \eqref{assumption3} with constants $\cAone = \beps \check{c}_1$, $\cAtwo = \dfrac1\beps \check{c}_2$, and $\cAthree = \dfrac1\beps \check{c}_3$, respectively.
\\ \indent
Furthermore, by fixing $\tau$, note that the final time of verifying~\eqref{assumption2} no longer depends on $\eps$ for the Mittag-Leffler kernels with $a>b$. \\
\begin{center}
	\begin{tabular}[h]{|c|l|c|c|c|c|}
		\hline
		&&&&&\\[-3mm]
		Choice of $\eps$ &Examples of relevant kernels& order of $a,\,b$& \eqref{assumption1} &\eqref{assumption2}& \eqref{assumption3}\\[5pt]
		\hline\hline
		
		$\eps = \delta$ & $\eps  {\rt^{-\alpha}} g_{\alpha}$ & \_	& \cmark & {\cmark}  &\cmark  \\[3mm]
		
		"&$\eps \left(\frac{\rt}{\tau}\right)^{a-b}\frac{1}{\tau^b}t^{b-1}E_{a,b}\left(-\left(\frac{t}{\tau}\right)^a\right) $	& $a \leq b$	& \cmark  & \cmark&\cmark \\[3mm]
		
		"&$\eps \left(\frac{\rt}{\tau}\right)^{a-b}\frac{1}{\tau^b}t^{b-1}E_{a,b}\left(-\left(\frac{t}{\tau}\right)^a\right) $	& $a > b $	& \cmark  & \cmark&\cmark\\[2mm]
		\hline
		&&&&&\\[-3mm]
		$\eps = \tau$ &$\delta \left(\frac{\rt}{\eps}\right)^{a-b}\frac{1}{\eps^b}t^{b-1}E_{a,b}\left(-\left(\frac{t}{\eps}\right)^a\right)$	& $a \leq b$	& \cmark  & \cmark&\cmark \\[4mm]
		
		"&$\delta \left(\frac{\rt}{\eps}\right)^{a-b}\frac{1}{\eps^b}t^{b-1}E_{a,b}\left(-\left(\frac{t}{\eps}\right)^a\right)$	&$a>b$  & \cmark & \sometimes &\cmark \\[2mm]
		\hline
	\end{tabular}
	~\\[2mm]
	\captionof{table}{Kernels for  flux laws discussed in Section~\ref{Sec:AcousticModeling} and the assumptions they satisfy.} \label{tab:ker_assu_veri}  
\end{center}

\subsection{Limiting behavior of the Kuznetsov equation for relevant classes of kernels}\label{Sec:limits_Kuzn}
We next wish to determine the limiting behavior of solutions to \eqref{ibvp_Kuzn_general} as $\eps \searrow 0$. 
The key result towards this is proving the continuity of the solution with respect to the memory kernel. 
\begin{theorem}
	\label{Thm:Limit}
	Let $\eps_1, \eps_2 \in (0, \beps)$. Under the assumptions of Theorem~\ref{Thm:Wellp_2ndorder_nonlocal_Kuzn},  for sufficiently small $T$ (independent of $\eps$), the following estimate holds:
	\begin{equation} \label{continuity_kernels}
		\|\psi^{\eps_1}-\psi^{\eps_2}\|_{\textup{E}}\lesssim \|(\frakK_{\eps_1}-\frakK_{\eps_2})\Lconv1\|_{L^1(0,T)},
	\end{equation}
	where $\|\cdot\|_{\textup{E}}$ is defined in~\eqref{energy_norm}.
\end{theorem}
\begin{proof}
	We note that the difference $\opsi=\psi^{\eps_1}-\psi^{\eps_2}$  solves the equation
	\begin{equation}
		\begin{aligned}\label{eq:equation_frakeps_diff_Kuzn}
			&(1+ 2\tk \psi_t^{\eps_1} )\opsi_{tt}-c^2 \D \opsi-  \frakK_{\eps_1}* \D\opsi_{t}\\
			=&\, \begin{multlined}[t] -2\tk \opsi_t \psi^{\eps_2} _{tt}- 2\ell \left( \nabla \psi^{\eps_1} \cdot \nabla \opsi_t +\nabla \opsi \cdot \nabla \psi^{\eps_2}_t \right)+(\frakK_{\eps_1}-\frakK_{\eps_2})* \Delta \psi_t^{\eps_2}
			\end{multlined}
		\end{aligned}
	\end{equation}
	with zero boundary and initial conditions. We can test \eqref{eq:equation_frakeps_diff_Kuzn} by $\opsi_{t}$ and proceed similarly to the the derivation of the lower-order estimate in \eqref{uniform_lower_order_Kuzn}. We therefore mainly discuss the contribution of the right-hand side in \eqref{eq:equation_frakeps_diff_Kuzn}.
	To this end, the resulting $\tk$ term after testing can be estimated as follows:
	\begin{equation}
		\begin{aligned}
			\left |-2 \tk \int_0^t \intO \opsi_t^2 \psi_{tt}^{\eps_2} \dxs \right| 
			\lesssim\,  \|\psi _{tt}^{\eps_2}\|_{L^1(L^\infty(\Om))}\|\opsi_t\|_{L^\infty_t(L^2(\Om))}^2. 
		\end{aligned}
	\end{equation} 
	We can estimate the $\ell$ term similarly:
	\begin{equation}
		\begin{multlined}
			\left|-2\ell\int_0^t \int_{\Omega}\left( \nabla \psi^{\eps_1} \cdot \nabla \opsi_t +\nabla \opsi \cdot \nabla \psi^{\eps_2}_t \right) \opsi_t \dxs \right| \lesssim \|\nabla \psi^{\eps_1}_t\|_{L^1(L^\infty(\Om))} \|\nabla \opsi_t^\eps\|^2_{L^\infty_t(L^2(\Om))} \\+ \|\nabla \psi^{\eps_2}_t\|_{L^1(L^\infty(\Om))} \left(\|\opsi_t\|^2_{L^\infty_t(L^2(\Om))}+\|\nabla \opsi\|_{L^\infty_t(L^2(\Om))}^2\right).
		\end{multlined}
	\end{equation}
	The $(\frakK_{\eps_1}-\frakK_{\eps_2})$ term we first rewrite as 
		\begin{equation}\label{psiepsdiff}
	\begin{aligned}
	&\int_0^t \left((\frakK_{\eps_1}-\frakK_{\eps_2})* \Delta \psi_t^{\eps_2}, \opsi_t\right)_{L^2(\Om)}\ds \\
	=&\, \int_0^t \big((1*(\frakK_{\eps_1}-\frakK_{\eps_2})*\Delta \psi_{tt}^{\eps_2})+(1*(\frakK_{\eps_1}-\frakK_{\eps_2}))\Delta \psi_1, \opsi_t(s)\big)_{L^2(\Om)}\ds .
	\end{aligned}
	\end{equation}
	Then it can be controlled using Young's convolution inequality as follows: 
	\begin{equation}\label{psiepsdiff}
		\begin{aligned}
			&\left |\int_0^t \left((\frakK_{\eps_1}-\frakK_{\eps_2})* \Delta \psi_t^{\eps_2}, \opsi_t\right)_{L^2}\ds \right|			
			\\ 
			\leq&\,\begin{multlined}[t]  \|(\frakK_{\eps_1}-\frakK_{\eps_2})*1\|_{L^1(0,T)} \left\{ \|\Delta \psi_{tt}^{\eps_2}\|_{L_t^1( L^2(\Om))}+
				\|\Delta \psi_1\|_{L^2(\Om)}\right\} \|\opsi_t\|_{L_t^\infty( L^2(\Om))}. \end{multlined}
			\\ \leq&\,\begin{multlined}[t]  
				\|\Delta \psi_1\|_{L^2(\Om)}\left( \frac1{4\gamma}\|(\frakK_{\eps_1}-\frakK_{\eps_2})*1\|_{L^1(0,T)}^2 + \gamma \|\opsi_t\|_{L_t^\infty( L^2(\Om))}^2\right)\\
				+\frac12 \|\Delta \psi_{tt}^{\eps_2}\|_{L_t^1( L^2(\Om))}\left( \|(\frakK_{\eps_1}-\frakK_{\eps_2})*1\|_{L^1(0,T)}^2 + \|\opsi_t\|_{L_t^\infty( L^2(\Om))}^2\right) \end{multlined}
		\end{aligned}
	\end{equation}
	for an arbitrary real number $\gamma>0$. In the proof of Theorem~\ref{Thm:Wellp_2ndorder_nonlocal_Kuzn} we have shown that $\psi^{\eps_1}$, $\psi^{\eps_2}$ 
	$\in \calB$, with $\calB$ defined in \eqref{defB_Kuzn}, which crucially implies uniform boundedness of $\|\Delta \psi_{tt}^{\eps_2}\|_{L_t^1( L^2(\Om))}$. Therefore, we can use the properties of $\mathcal{B}$ to derive the following estimate:
	\begin{equation} 
		\begin{aligned}
			&\frac12 \underline{\aaa}\sup_{\sigma\in(0,t)}\nLtwo{\opsi_t(\sigma)}^2+\frac{c^2}{2}\sup_{\sigma\in(0,t)}\nLtwo{\nabla \opsi(\sigma)}^2\\
			\lesssim&\,\begin{multlined}[t] C(\Omega)(TR+\gamma \|\D \psi_1\|_{L^2(\Om)})\left(\|\opsi_t\|_{L_t^\infty(L^2(\Om))}^2 
				+ \|\nabla \opsi\|_{L_t^\infty( L^2(\Om))}^2 \right)\\+ (TR+\frac1{4\gamma}\|\D \psi_1\|_{L^2(\Om)}) \|(\frakK_{\eps_1}-\frakK_{\eps_2})*1\|_{L^1(0,T)}^2. \end{multlined}
		\end{aligned}
	\end{equation}
By reducing the final time $T$ and $\gamma$ if needed (independently of $\eps$), we can absorb the energy term on the right-hand-side of the inequality which concludes the proof.
\end{proof}
 Note that using Theorem~\ref{Thm:Limit} to compute the limiting behavior of $\peps$ relies on the continuity of the kernel $\frakKeps$ with respect to the parameter $\eps$ in a weaker norm compared to $\|\cdot\|_{L^1(0,T)}$. This allows us to consider $\delta_0$ as a possible limit as well as to establish rates of convergence for the Mittag-Leffler kernels by exploiting their completely monotone properties; see, e.g., Proposition~\ref{Prop:Limit_a<=b} for the usefulness of this result in obtaining a convergence rate.
 
 The convergence rate of the solutions will depend on the form of the kernel $\frakKeps$ and its specific dependence on $\varepsilon$. We therefore treat the vanishing sound diffusivity and relaxation time limits separately.

\subsubsection{The vanishing sound diffusivity limit}
We first discuss the setting $\frakKeps=\eps \frakK$. Recall that  
two important examples
of this class of kernels (up to a constant) 
are
\begin{equation}
	\frakKeps= \eps \frakK \quad \text{with } \ 
	\frakK(t)= \left\{
	\begin{aligned}
		& 	\rt^{-\alpha} g_{\alpha}(t)\\
		& \textrm {or} \\[-2mm]
		& \left(\frac{\rt}{\tau}\right)^{a-b}\frac{1}{\tau^b}t^{b-1}E_{a,b}\left(-\left(\frac{t}{\tau}\right)^a\right) \ \textrm{with $\tau$ fixed}
	\end{aligned}
	\right.
\end{equation}
\noindent where the Abel kernel $g_\alpha$ is defined in \eqref{def_galpha} for $\alpha \in (0,1)$ and $g_0= \delta_0$. The Mittag-Leffler kernel with $0< a,b\leq1$ is defined in~\eqref{ML_kernels}. Here $\eps$ holds the place of the sound diffusivity $\delta$. 

\begin{corollary}\label{Corollary:Limit_epsK}
	Under the assumptions of Theorems~\ref{Thm:Wellp_2ndorder_nonlocal_Kuzn} and~\ref{Thm:Limit} with the kernel
	\[
	\frakKeps = \eps \frakK, \quad \eps \in (0, \beps)
	\] 
	satisfying assumptions \eqref{assumption1}, \eqref{assumption2}, and \eqref{assumption3}, the family of solutions $\{\peps\}_{\eps \in (0, \beps)}$ of \eqref{ibvp_Kuzn_general}  converges in the energy norm to the solution $\psi$ of the inviscid Kuznetsov equation 
	\begin{equation}\label{ibvp_West_limit}
		\left \{	\begin{aligned} 
			&(1+2\tk\psi_t)\psi_{tt}-c^2 \Delta \psi + 2\ell\, \nabla \psi\cdot\nabla \psi_t  = 0 \quad  &&\text{in } \Omega \times (0,T), \\
			&\psi =0 \quad  &&\text{on } \partial \Omega \times (0,T),\\
			&(\psi, \psi_t)=(\psi_0, \psi_1), \quad  &&\text{in }  \Omega \times \{0\},
		\end{aligned} \right.
	\end{equation}
	at a linear rate
	\begin{equation}
		\|\peps-\psi\|_{\textup{E}}\lesssim \eps
		\quad \mbox{ as } \ \eps  \searrow 0.
	\end{equation}
\end{corollary}
\begin{proof}
	In this setting, the limiting kernel is $\frakK_0=0$, and it satisfies assumptions \eqref{assumption1}, \eqref{assumption2}, and \eqref{assumption3}. By Theorem~\ref{Thm:Limit}, we then immediately have
	\begin{equation}
		\begin{aligned}
			\|\peps-\psi\|_{\textup{E}} \leq C \|\frakKeps*1\|_{L^1(0,T)} = C \eps \|\frakK*1\|_{L^1(0,T)}, 
		\end{aligned}
	\end{equation}
	for some $C>0$, independent of $\eps$, which concludes the proof.
\end{proof}

\subsubsection{The vanishing thermal relaxation time limit with Mittag-Leffler kernels} \label{Sec:tau_limit_Kuzn}
We now turn our attention to the kernels that were motivated by the presence of thermal relaxation in the heat flux laws of the propagation medium, and have the form
\begin{equation} \label{ML_kernel_Sec4}
	\begin{aligned}
		\frakKeps(t)=&\,\delta\left(\frac{\rt}{\eps}\right)^{a-b}\frac{1}{\eps^b}t^{b-1}E_{a,b}\left(-\left(\frac{t}{\eps}\right)^a\right), \quad a,b \in (0,1].
	\end{aligned}
\end{equation}
\indent In what follows, we intend to take the limit $\eps \searrow0$, while keeping $\rt>0$ fixed. 
Unlike our work in~\cite{kaltenbacher2023limting},
on the fractionally damped Westervelt equation, we limit ourselves here to the case $a-b \leq 0$ (since, to the best of our knowledge, \eqref{assumption2} does not hold uniformly in $\eps$ when $a-b>0$). We will then prove that solutions $\peps$ of \eqref{ibvp_Kuzn_general} converge to the solution $\psi$ of the following time-fractional equation: 
\begin{equation} 
	\begin{aligned}
		&	(1+2{k} \psi_t)\psi_{tt}-c^2 \Delta \psi -  \rt^{a-b}
		{\textup{D}_t^{a-b+1}\D u} + 2\ell \,\nabla\psi\cdot\nabla\psi_t =f,	\end{aligned}
\end{equation}
supplemented by the same boundary and initial conditions as in \eqref{ibvp_Kuzn_general}. Recall that 
$$\textup{D}_t^{a-b+1}\D \psi = g_{b-a} \Lconv \D \psi_t. $$ Note also that in case $a=b$, the limiting equation is strongly damped. 
\begin{proposition}\label{Prop:Limit_a<=b}
	Let $\rt>0$ and $\delta>0$ be fixed, and let the assumptions of Theorems~\ref{Thm:Wellp_2ndorder_nonlocal_Kuzn} and~\ref{Thm:Limit} hold. Consider the family of solutions $\{\peps\}_{\eps \in(0,\beps)}$ of \eqref{ibvp_Kuzn_general} with the kernel given by
	\begin{equation} \label{kernel_form}
		\begin{aligned}
			\frakKeps(t)=&\,\delta\left(\frac{\rt}{\eps}\right)^{a-b}\frac{1}{\eps^b}t^{b-1}E_{a,b}\left(-\left(\frac{t}{\eps}\right)^a\right) \quad \text{where } \, 0 < a \leq b\leq 1.
		\end{aligned}
	\end{equation}
Then the family	 $\{\peps\}_{\eps \in(0,\beps)}$
	converges to the solution $\psi$ of 
	\begin{equation}
		\left \{	\begin{aligned} 
			&(1+2{k} \psi_t)\psi_{tt}-c^2 \Delta \psi -  \D\frakK_0\Lconv\psi_t + 2\ell \,\nabla\psi\cdot\nabla\psi_t = 0\quad  &&\text{in } \Omega \times (0,T), \\
			&\psi =0 \quad  &&\text{on } \partial \Omega \times (0,T),\\
			&(\psi, \psi_t)=(\psi_0, \psi_1), \quad  &&\text{in }  \Omega \times \{0\},
		\end{aligned} \right.
	\end{equation}
	with the kernel $\frakK_0= \delta\rt^{a-b} g_{b-a}$ at the following rate:
	\begin{equation}
		\|\peps-\psi\|_{\textup{E}}\lesssim \|(\frakKeps-\frakK_0)\Lconv1\|_{L^1(0,T)}  \sim\eps^a
		\quad \mbox{ as } \ \eps  \searrow 0.
	\end{equation}
\end{proposition}
\begin{proof}
	By Theorem~\ref{Thm:Limit} ($\frakK_0$ satisfies \eqref{assumption1}, \eqref{assumption2}, and \eqref{assumption3}), we have
	\begin{equation}
		\| \peps-\psi\|_{\textup{E}}\leq C  \|(\mathfrak{K}_{\eps}-\mathfrak{K}_{0})*1\|_{L^1(0,T)}
	\end{equation}
	To further establish the asymptotic behavior of the right-hand side as $\eps \searrow 0$, we rely on the following asymptotic behavior established in~\cite[Proposition 4.1]{kaltenbacher2023limting}:
	\[
	\begin{aligned}
		\|(\frakKeps-\frakK_0)*1\|_{L^1(0,T)}=&\, \delta T^{1+b-a}E_{a,2+b-a}\left(-\left(\frac{T}{\eps}\right)^a\right)\\
		\sim&\, \delta \frac{T^{1+b-a}}{\Gamma(2+b-2a)} \left(\frac{T}{\eps}\right)^{-a} \qquad
		\text{ as }\ \frac{T}{\eps}\to\infty,
	\end{aligned}
	\] 
	which yields the claimed rate of convergence when $0 < a \leq b\leq 1$.
\end{proof}

\section{Differences in the analysis of the nonlocal Blackstock equation }\label{Sec:Blackstock_prop}
We now turn our attention to the limiting behaviour of the Blackstock equation and consider the following initial-boundary value problem: 
\begin{equation}\label{ibvp_Blackstock_general}
	\left \{	\begin{aligned} 
		&\pepstt-c^2 (1-2\tk \pepst)\Delta \peps - \frakKeps * \Delta \pepst + 2 \ell \,\nabla \peps \cdot \nabla\pepst= 0 \quad  &&\text{in } \Omega \times (0,T), \\
		&\peps =0 \quad  &&\text{on } \partial \Omega \times (0,T),\\
		&(\peps, \pepst)=(\psi_0, \psi_1), \quad  &&\text{in }  \Omega \times \{0\}.
	\end{aligned} \right.
\end{equation}
We focus in this section on pointing out the differences in treating the Blackstock nonlinearity compared Kuznetsov's. First off, the coercivity assumption \eqref{assumption3} can be weakened. 
\subsection*{Blackstock-specific assumption on the memory kernel} 
 In the absence of a nonlinearity involving $\peps_{tt}$, a weaker coercivity assumption on the kernel  compared to what we had in \eqref{assumption3} suffices. \vspace*{2mm}
\begin{center}
	\fbox{ 
		\begin{minipage}{0.8\textwidth}
			It holds 
			\begin{equation}\label{assumption3_Black} \tag{\ensuremath{\bf {{A}}^{\textbf{B}}_3}}
				 \int_0^t\bigl((\frakKeps\Lconv y)(s), y(s)\bigr)_{L^2(\Omega)}\ds \geq 0, \quad 
				y\in L^2(0,t;L^2(\Omega))    
			\end{equation}
			for all $t\in(0,T)$.
		\end{minipage}
	}
\end{center}
\vspace*{2mm}
Since \eqref{assumption3_Black} assumption is weaker than \eqref{assumption3}, we refer to Section~\ref{Sec:Assumption_Kuzn} and \cite[Section 5.2]{kaltenbacher2023limting} for discussions on the verification of this assumption. Assumption~\eqref{assumption3_Black} is in particular verified for all the kernels considered in this work.

\subsection{Uniform well-posedness of the nonlocal Blackstock equation}
The main result to prove in order to conduct the limiting behavior analysis is to show uniform well-posedness of \eqref{ibvp_Blackstock_general} with respect to $\eps$. The strategy will rely on studying a linearized problem combined with a Banach fixed-point argument, as before. However, here a different linearization is needed to tackle the Blackstock nonlinearity, namely
\begin{subequations} \label{ibvp_2ndorder_lin_Blackstock}
	\begin{equation} \label{ibvp_2ndorder_lin_Blackstock:Eq}
		\begin{aligned}
			 \peps_{tt}-c^2 \bbb\Delta \peps - \frakKeps* \D\peps_{t}+ \nabla \lll \ \cdot \nabla \peps_t=0 \ \text{in }\Omega\times(0,T),
		\end{aligned} 
	\end{equation}
	with variable coefficients
	\begin{align}
		\bbb=1-2\tk \phi_t, \qquad \lll= 2 \ell \phi, \qquad \tk, \ell \in \R,
	\end{align}
	supplemented by the initial and boundary conditions:
	\begin{equation} \label{ibvp_2ndorder_lin_Blackstock:Conditions}
		\begin{aligned}
			(\peps, \peps_t) \vert_{t=0}=\,(\psi_0, \psi_1), \qquad \peps \vert_{\partial \Omega}=0.
		\end{aligned} 
	\end{equation}
\end{subequations}

\subsubsection{ Uniform well-posedness of a linearized Blackstock equation with variable coefficients}
Compared to Proposition~\ref{Prop:WellP_Lin_2ndorder_nonlocal}, we need less regularity-in-time from our solution. This can be explained by the fact that the term $\pepstt$ appears linearly in the Blackstock equation. This leads us to consider a different space of solutions and bootstrap argument.
\begin{proposition}
	\label{Prop:WellP_Lin_2ndorder_nonlocal_Blackstock}
	Let $\eps \in (0, \beps)$ and $k$, $\ell \in \R$. Given $T>0$, 
	let  $\phi \in X_\phi$ and assume that there exist $\overline{\bbb}$ and $\underline{\bbb}$, independent of $\eps$, such that
	\begin{equation} \label{non-deg_phi_Black}
		0<\underline{\bbb} \leq \bbb(\phi)=1-2\tk \phi_t(x,t)\leq \overline{\bbb} \quad \text{a.e. in } \ \Omega \times (0,T). 
	\end{equation} Let  assumptions \eqref{assumption1}, \eqref{assumption2}, and \eqref{assumption3_Black} on the kernel hold.
	Furthermore, assume that the initial conditions satisfy
	\begin{equation}
			(\psi_0, \psi_1) \in \begin{cases}
				\Honefour\times \Honethree \ \text{if } \frakKeps \equiv \eps \delta_0 \ \text{or }  \frakKeps \equiv 0,\\[2mm]
				\Honefour\times \Honefour \ \text{if } \frakKeps \not \equiv 0 \in L^1(0,T) . 
			\end{cases}
	\end{equation}
	Then there exists $\Xi^\textup{\textbf{B}}>0$, independent of $\eps$, such that if
	\begin{equation}\label{def_xi_Black}
		\begin{aligned}
			\begin{multlined}[t] 
				\|\phi\|_{L^1(H^4(\Omega))} +	\|\phi_{t}\|_{L^1(H^3(\Omega))} +	\|\phi_{tt}\|_{L^2(H^2(\Omega))} \leq \Xi^\textup{\textbf{B}},
			\end{multlined}
		\end{aligned}
	\end{equation}
	and if the final time $T=T(\|\phi\|_{X_\phi})$ is small enough, then there is a unique  solution $\peps$ of \eqref{ibvp_2ndorder_lin_Blackstock:Eq}, \eqref{ibvp_2ndorder_lin_Blackstock:Conditions} in
	\begin{equation}\label{solution_space_psi_Black}
		\begin{aligned}
			X_{\psi}^\textbf{B} =& \,\begin{multlined}[t] W^{2,\infty}(0,T;\Honetwo)\cap W^{1,\infty}(0,T;\Honethree) \cap L^2(0,T;\Honefour).
			\end{multlined}
		\end{aligned}
	\end{equation} 
	This solution satisfies the estimate 
	\begin{equation} \label{Lin2ndorderNonlocal_Kuzn:Main_energy_est_Black}
		\begin{aligned}
			& \begin{multlined}[t] \|\peps\|^2_{X_\psi^\textbf{B}}
			\end{multlined}
			\leq C_{\textup{lin}}^\textbf{B}(T)\Bigl(\|\psi_0\|^2_{H^4(\Om)}+\|\psi_1\|^2_{\Hthree}+ \eta\|\frakKeps\|_{L^1(0,T)} ^2\|\D^2 \psi_1\|^2_{\Ltwo}
			{\Bigr)}.
		\end{aligned}
	\end{equation}
	where $\eta = 0$ if $\frakKeps \in \{ 0, \eps\delta_0 \}$ and $\eta =1$ otherwise.
\end{proposition}
\begin{proof} 
	Similarly to the approach in the proof of Proposition~\ref{Prop:WellP_Lin_2ndorder_nonlocal}, we use a Galerkin procedure in a suitable Hilbert space and consider a time-differentiated semi-discrete problem. We set $p=\peps_t$, where, again, the dependence of $p$ on $\eps$ is omitted to simplify the notation. 
	The time-differentiated semi-discrete nonlocal Blackstock equation is then given by
		\begin{equation} \label{time_diff_discr_2ndorder_nonlocal}
		 p_{tt}-c^2 \bbb\Delta p -   \frakKeps* \D p_{t}= \fB 
	\end{equation}
	with the right-hand side
	\begin{equation}\label{eq:def_f_Black}
		\fB(t)= -\nabla \lll_t \cdot \nabla p-\nabla \lll \cdot \nabla p_t 
	+c^2 \bbb_t\Delta \peps + \eta \frakKeps(t) \D p(0).
	\end{equation}
	Similarly to the proof of Proposition~\ref{Prop:WellP_Lin_2ndorder_nonlocal}, we test this equation with $(-\Delta)^2p_t$ and integrate over space and $(0,t)$, where now we can rely on assumption \eqref{assumption3_Black} to conclude that
	\begin{align} \label{nonngeg_kernel}
\int_0^t\bigl((\frakKeps\Lconv \nabla \Delta p_t)(s), \nabla \Delta p_t(s)\bigr)_{L^2(\Omega)}\ds \geq 0	.
	\end{align} Most of the estimates remain the same, we therefore simply point out how to deal with the $\bbb$ terms arising after integration by parts in \eqref{ineq_1} and \eqref{eq:bootstrap_Kuzn}. We first note that, when $\bbb \neq 1$, \eqref{c_term_est} generalizes to
	\begin{equation}
		\begin{aligned}
			&-c^2(\bbb \Delta p, (-\Delta)^2 p_{t})_{L^2}\\
			=&\, \frac{c^2}{2}\ddt  (\bbb \nabla \Delta p, \nabla \Delta p)_{L^2} -\frac{c^2}{2} (\bbb_t \nabla \Delta p, \nabla \Delta p)_{L^2}-c^2(\nabla \Delta p \cdot \nabla \bbb+\D p \D \bbb,  \Delta p_t)_{L^2}.
		\end{aligned}
	\end{equation}
	Similarly to before, by the time-differentiated semi-discrete PDE  \eqref{time_diff_discr_2ndorder_nonlocal}, we have $\fB=0$ on the boundary and thus
	$(\fB, (-\Delta)^2p_t)_{L^2}=(\Delta \fB, \Delta p_t)_{L^2}$.
	Testing \eqref{time_diff_discr_2ndorder_nonlocal} with $(-\Delta)^2 p_t$ and utilizing \eqref{nonngeg_kernel} therefore yields
	\begin{equation}\label{enid_2ndorder_lin}
		\begin{aligned}
			&\frac12 \nLtwo{\D p_t}^2 \big \vert_0^t
			+\frac{c^2}{2} \nLtwo{\sqrt{\bbb}\nabla\Delta p}^2\big \vert_0^t  \\
			\leq&\, \begin{multlined}[t] \int_0^t\Bigl(\frac{c^2}{2}(\bbb_t\nabla\D p,\nabla\D p)_{L^2}
				+(c^2\nabla\D p\cdot\nabla\bbb + c^2\D p\, \D\bbb +\D \fB,\D p_t)_{L^2}
				\Bigr)\ds. \end{multlined}
		\end{aligned}
	\end{equation}
	The terms on the right-hand side of \eqref{enid_2ndorder_lin} can be estimated as follows. We first have
	\[\int_0^t
	\frac{c^2}{2}(\bbb_t\nabla\D p,\nabla\D p)_{L^2}
	\ds
	\lesssim\, \Xi^\textbf{B} \|\nabla\D p\|_{L^\infty_t(L^2(\Om))}^2.
	\]
	Additionally,
	\[
	\begin{aligned}
		&\int_0^t
		(c^2\nabla\D p\cdot\nabla\bbb + c^2\D p\, \D\bbb,\D p_t)_{L^2}
		\ds
		\leq 
		c^2 \| \nabla\D p\cdot\nabla\bbb +\D p\, \D\bbb \|_{L^1(L^2(\Om))}
		\|\D p_t\|_{L^{\infty}_t(L^2(\Om))}.
	\end{aligned}
	\]
	The first term on the right in the above bound can be further estimated as follows:
	\[
	\begin{aligned}
		&\|\nabla\D p\cdot\nabla\bbb+\D p\, \D\bbb\|_{L^1(L^2(\Om))}\\
		\lesssim&\,	\|\nabla\D p\|_{L^\infty_t(L^2(\Om))}\|\nabla\bbb\|_{L^1(L^\infty(\Om))}
		+\|\nabla \D p\|_{L^\infty_t(L^2(\Om))}\|\D\bbb\|_{L^1(L^3(\Om))}\\
		\lesssim&\,	\Xi^\textbf{B}(\|\nabla\D p\|_{L^\infty_t(L^2(\Om))}
		+\|\nabla \D p\|_{L^\infty_t(L^2(\Om))})
	\end{aligned}
	\]
	since $\D p=0$  on $\partial \Omega$. 
It thus remains to estimate the term $\int_0^t
(\D \fB,\D p_t)_{L^2}
\ds.$
	Recalling how $\fB$ is defined in \eqref{eq:def_f_Black},
	we have the following identity: 
	\begin{equation} \label{Delta_tildef_Black}
		\begin{aligned}
			&\D\fB
			=\, \begin{multlined}[t]
				c^2 \D\bbb_t\D \peps+ 2c^2 \nabla\bbb_t\cdot\nabla\D \peps +c^2 \bbb_t\Delta^2 \peps 
				\\+\Delta(-\nabla \lll_t \cdot \nabla p-\nabla \lll \cdot \nabla p_{t} )+ \eta \frakKeps(t) \D^2 p(0), \end{multlined}
		\end{aligned}
	\end{equation}
We use the same splitting approach as in the proof of Proposition~\ref{Prop:WellP_Lin_2ndorder_nonlocal}
	\begin{equation} \label{Delta_tildef_Deltapt_second_Black}
		\begin{aligned}
			\int_0^t (\D \fB,\D p_t)_{L^2} \ds
			\leq&\,  
			\|\D \fB+g\|_{L^1(L^2(\Om))} \|\D p_t\|_{L^{\infty}_t(L^2(\Om))}
			+\left|\int_0^t (g,\D p_t)_{L^2} \ds\right|,
		\end{aligned}
	\end{equation}
	with $g=\nabla\lll\cdot\nabla \D p_t$. Then, since $\Delta p_t=0$ on $\partial \Omega$, we can estimate (see inequality \eqref{ineq:g_estimate})
	\begin{equation} 
	\begin{aligned}
		\left|\int_0^t (g,\D p_t)_{L^2} \ds\right|
		\lesssim \;\Xi^\textbf{B} \|\D p_t\|_{L^\infty_t(\Ltwo)}^2.
	\end{aligned}
\end{equation}
	We further have
	\begin{equation}  \label{est_Delta_tildef_Black}
		\begin{aligned}
		&	\|\D\fB+g\|_{L^1(L^2(\Om))}\\
			\lesssim&\, \begin{multlined}[t]
			\|\D\bbb_t\|_{L^1(L^2(\Om))}  {\|\D \peps\|_{L^{\infty}_t(L^\infty(\Om))}}
					+\|\nabla\bbb_t\|_{L^1(L^3(\Om))} {\|\nabla\D \peps\|_{L^{\infty}_t(L^6(\Om))}}
			\\	+\|\bbb_t\|_{L^2(L^\infty(\Om))} \|\D^2 \peps\|_{L^{2}(L^2(\Om))}
				+\|\nabla \D \lll_t \cdot \nabla p \|_{L^1(L^2(\Om))}\\+\|\nabla \lll_t \cdot \nabla \D p \|_{L^1_t(L^2(\Om))}+\|D^2 \lll_t : D^2 p \|_{L^1_t(L^2(\Om))}\\
				+\|\nabla \D \lll \cdot \nabla p_t \|_{L^1_t(L^2(\Om))}
				+\|D^2 \lll: D^2 p_t \|_{L^1_t(L^2(\Om))}
				+ \eta \|\frakKeps\|_{L^1(0,t)} \|\D^2 \psi_1\|_{L^2(\Om)}. \end{multlined}
		\end{aligned}
	\end{equation}
The $\lll$-terms are then estimated similarly to before (see \eqref{estlll1}, \eqref{estlll2}, and \eqref{estlll3_Honefour}) to obtain 
	\begin{equation}  \label{est_Delta_tildef_second_Black}
		\begin{aligned}
			&\int_0^t (\D \fB,\D p_t)_{L^2} \ds \\
			\lesssim&\, \begin{multlined}[t]
				\Xi^\textbf{B} \left(\|\D p_t\|_{L^\infty_t(L^2(\Om))}^2+\|\nabla \D p \|^2_{L^\infty_t(L^2(\Om))} \right)
				\\	
				+\|\bbb_t\|_{L^2(H^2(\Om))} {\|\D^2 \peps\|_{L^{2}(L^2(\Om))}}
				+ \eta \|\frakKeps\|_{L^1(0,t)} \|\D^2 \psi_1\|_{L^2(\Om)}. \end{multlined}
		\end{aligned}
	\end{equation}
	
	\noindent \emph{Bootstrap argument for $\peps \in L^2(0,T;\Honefour)$}. Similarly to Proposition~\ref{Prop:WellP_Lin_2ndorder_nonlocal}, if $\ell \not=0$, but now also if $\bbb_t\not=0$, we have to estimate the $\Delta^2 \peps$ term in \eqref{est_Delta_tildef_second_Black}. To this end, we test 
	\begin{align} \label{bootstrap_identity_Black}
		[\frakKeps\Lconv\partial_t+\bbb c^2]\D\peps =r:= p_t+\nabla \lll \cdot \nabla p
	\end{align}
	with 
	$\D^3\peps$ 
	and integrating by parts, due to assumption \eqref{assumption2}, we have
	\[
	\begin{aligned}
		&c^2\|\sqrt{\bbb}\D^2\peps\|_{L^2_t(L^2(\Om))}^2\\
		\leq&\,
		\int_0^t\Bigl\{(\D r,\D^2\peps)_{L^2}-c^2\D[\bbb\D\peps],\D^2\peps)_{L^2}\Bigr\}\ds + \cAtwo \|\Delta^2 \psi_0\|^2_{\Ltwo}\\ 
		& \begin{multlined}
			\leq\|\D^2\peps\|_{L^2_t(L^2(\Om))}
		\left\|\D p_t-\D [\nabla \lll \cdot \nabla p]-c^2\D\bbb\D\peps-c^2\nabla\bbb\nabla\D\peps \right\|_{L^2_t(L^2(\Om))}  + \cAtwo \|\psi_0\|^2_{H^4(\Om)},
			\end{multlined}
	\end{aligned}
	\]
	where we have used the fact that $r$ and $\D^2 \peps$ vanish on the boundary.
	Contrary to the proof of Proposition~\ref{Prop:WellP_Lin_2ndorder_nonlocal}, here we do not have access to an intermediate estimate ({\it cf.} \eqref{ineq:intermediate}), therefore we don't know {\it a priori} if $\D p_t-\D [\nabla \lll \cdot \nabla p]-c^2\D\bbb\D\peps-c^2\nabla\bbb\nabla\D\peps \in L^2(0,T;L^2(\Om))$. Instead, we use the following inequalities to estimate the corresponding $L^2(0,T; L^2(\Om))$-norm:
	\begin{equation}
		\begin{aligned}
			\|\D [\nabla \lll \cdot \nabla p]\|_{L_t^2(L^2(\Om))} \lesssim &\; \|\lll\|_{L_t^2(H^3(\Om))} \|\nabla\Delta p\|_{L_t^\infty(L^2(\Om))},
		\end{aligned}
	\end{equation}
	together with
	\begin{equation}\label{ineq:bootstrap_2}
		\begin{aligned}
	 		\|\D\bbb\D\peps\|_{L_t^2(L^2(\Om))} \lesssim &\; \|\bbb\|_{L^2_t(H^3(\Omega))} \|\nabla\D \peps\|_{L^\infty_t(L^2(\Omega))} \\ 
	 		\lesssim&\; \|\bbb\|_{L^2_t(H^3(\Omega))} (\|\nabla\D p\|_{L^1_t(L^2(\Omega))} + \|\nabla \D \psi_0\|_{H^3(\Om)})\\ 
	 		\lesssim&\; \|\bbb\|_{L^2_t(H^3(\Omega))} (T\|\nabla\D p\|_{L^\infty_t(L^2(\Omega))} + \|\nabla \D \psi_0\|_{H^3(\Om)}),
	 	\end{aligned}
	\end{equation}
	and, since $\Delta \psi \vert_{\partial \Omega}=0$, by elliptic regularity
	\begin{equation}
		\begin{aligned}
			\|\nabla\bbb\nabla\D\peps\|_{L_t^2(L^2(\Om))} \lesssim &\; \|\bbb\|_{L_t^\infty(H^3(\Om))} \|\nabla\D\peps\|_{L_t^2(L^2(\Om))}\\
			\lesssim &\; \|\bbb\|_{L_t^\infty(H^3(\Om))} \|\D^2\peps\|_{L_t^2(L^2(\Om))}.
		\end{aligned}
	\end{equation}
	 We then rely on the smallness of $\|\bbb_t\|_{L^2(L^\infty(\Om))} \leq \Xi^\textbf{B}$ (and eventually of final time $T$), which multiplies $\|\D^2\peps\|_{L^2(L^2(\Om))}$ in \eqref{est_Delta_tildef_Black}, to absorb the arising terms and obtain \eqref{Lin2ndorderNonlocal_Kuzn:Main_energy_est_Black}. This yields the desired estimates.
\end{proof}
Note that in Proposition~\ref{Prop:WellP_Lin_2ndorder_nonlocal_Blackstock},
the need of small $T$ arises from estimate \eqref{ineq:bootstrap_2}. This requirement can be alleviated by establishing additional estimates by testing \eqref{ibvp_2ndorder_lin_Blackstock:Eq} by $\D^2\pepst$. To avoid increased technicality, and since our goal is ultimately to establish the well-posedness of the nonlinear equation~\eqref{ibvp_Blackstock_general}, we do not pursue this refinement. 
\\ \indent
Similarly to the Kuznetsov case, we can establish a uniform lower-order estimate for \eqref{ibvp_2ndorder_lin_Blackstock}.
Under the assumption of
Proposition~\ref{Prop:WellP_Lin_2ndorder_nonlocal_Blackstock}, testing \eqref{ibvp_2ndorder_lin_Blackstock} with $\psi_{t}$ and integrating over space and time yields, after usual manipulations, 
\begin{equation} \label{first_est_selfm_Black}
	\begin{aligned}
	\begin{multlined}
		\frac12 \left\{\|\pepst(t)\|^2_{L^2}+c^2\|\sqrt{\bbb} \nabla \peps(t)\|^2_{L^2} \right\}\Big \vert_0^t
		\leq \, -c^2 \int_0^t(\peps \nabla \bbb, \nabla \pepst)_{L^2}\ds\\+\frac12 c^2\|\bbb_t\|_{L^1(L^\infty(\Om))}\|\nabla \peps\|^2_{L^\infty(0,t;L^2(\Om))}  - \int_0^t( \nabla \lll \cdot \nabla \pepst, \pepst)\ds.  
	\end{multlined}
	\end{aligned}
\end{equation}
To bound the first term on the right, we use integration by parts in time and H\"older's inequality to obtain
\begin{equation}
	\begin{aligned}
		&-c^2 \int_0^t(\peps \nabla \bbb, \nabla \pepst)_{L^2}\ds\\
		\lesssim&\, \begin{multlined}[t]  \|\peps\|_{L^\infty_t(L^4(\Om))}\|\nabla \bbb\|_{L^\infty(L^4(\Om))} \|\nabla \peps\|_{L^\infty_t(L^2(\Om))}
			+ \|\pepst\|_{L^\infty_t(L^2(\Om))}\|\nabla \bbb\|_{L^1(L^\infty(\Om))} \|\nabla \peps\|_{L^\infty_t(L^2(\Om))}\\+ \|\peps\|_{L^\infty_t(L^6(\Om))}\|\nabla \bbb_t\|_{L^1(L^3(\Om))} \|\nabla \peps\|_{L^\infty_t(L^2(\Om))}.
		\end{multlined}
	\end{aligned}
\end{equation}
The last term is treated using estimate \eqref{ineq:lower_order_1}.
Note that, due to the embedding $W^{1,1}(0,T)\hookrightarrow L^\infty(0,T)$, we have
\begin{align}
	&\|\nabla \bbb\|_{L^\infty(L^4(\Om))} \lesssim \|\phi_{t}\|_{L^1(H^2(\Om))} + \|\phi_{tt}\|_{L^2(H^2(\Om))} \lesssim \Xi^\textbf{B}, 
	\shortintertext{and}
	&\|\nabla \bbb\|_{L^1(L^\infty(\Om))} \lesssim \|\phi_t\|_{L^1(H^3(\Om))} \lesssim \Xi^\textbf{B},
	\\
	& \|\nabla \bbb_t\|_{L^1(L^3(\Om))} \lesssim \|\phi_{tt}\|_{L^2(H^2(\Om))} \lesssim \Xi^\textbf{B}.
\end{align}
Thus, utilizing the above estimates and the smallness of $\Xi^\textbf{B}$ (and eventually final time $T$), leads to the uniform bound in $\eps$:
\[ \|\pepst(t)\|^2_{L^2(\Om)} + \| \nabla \peps(t)\|^2_{L^2(\Om)} 
\lesssim \|\psi_1\|^2_{L^2(\Om)} + \| \nabla \psi_0\|^2_{L^2(\Om)}
\]
a.e.\ in time. The next theorems and propositions are extensions of the results of Section~\ref{Sec:Kuznetsov_prop} to the Blackstock equation setting.

\subsection{Well-posedness and limiting behavior of the nonlocal Blackstock equation}\label{Sec:nonlin_wellp_cont_Black} The well-posedness of the nonlinear initial-boundary value problem~\ref{ibvp_Blackstock_general} relies again on setting up a fixed-point mapping
$
\TK^\textbf{B}:\phi \mapsto \peps$,
where $\phi$ will belong to a ball in a suitable Bochner space and $\peps$ solves the associated linearized problem~\eqref{ibvp_2ndorder_lin_Blackstock}. The proof is similar to that of Theorem~\ref{Thm:Wellp_2ndorder_nonlocal_Kuzn} so we omit it here.

\begin{theorem}\label{Thm:Wellp_2ndorder_nonlocal_Blackstock}
	Let $\eps \in (0, \beps)$ and $k$, $\ell \in \R$. Furthermore, let $(\psi_0, \psi_1) \in \Honefour \times \Honefour$ be such that
	\begin{equation}
		\|\psi_0\|^2_{H^4(\Om)}+\|\psi_1\|^2_{H^4(\Om)} \leq r^2,
	\end{equation}
	where $r$ does not depend on $\eps$. 
	Let assumptions \eqref{assumption1}, \eqref{assumption2}, and \eqref{assumption3_Black} on the kernel hold.
	Then, there exist a 
	data size $r_0=r_0(r)>0$, and final time $T=T(r)>0$, both independent of $\eps$, such that if 
	\begin{equation}\label{initsmallWellp_2ndorder_nonlocal_Black}
		\|\psi_0\|^2_{H^1(\Om)}+\|\psi_1\|^2_{L^2(\Om)} \leq r_0^2,
	\end{equation}
	then there is a unique solution $\psi$ in $X_\psi^\textbf{B}$ (defined in \eqref{solution_space_psi_Black}) of  equation \eqref{ibvp_Blackstock_general}, which satisfies the following estimate: 
	\begin{equation} \label{Nonlin:Main_energy_est_Black}
		\begin{aligned}
			 \|\peps\|^2_{X_\psi^\textbf{B}}
			\leq\, C_{\textup{nonlin}}^\textbf{B}\left(\,  \|\psi_0\|^2_{H^4(\Om)}+\|\psi_1\|^2_{H^3(\Om)}+ \eta \|\frakKeps\|_{L^1(0,T)} ^2\|\D^2 \psi_1\|^2_{L^2(\Om)}\right).
		\end{aligned}
	\end{equation} 
	where $\eta = 0$ if $\frakKeps \in \{ 0, \eps\delta_0 \}$ and $\eta =1$ otherwise.
	Here, $C_{\textup{nonlin}}^\textbf{B}=C_{\textup{nonlin}}^\textbf{B}(\Omega,\ophi, T)$ 
	does not depend on the parameter $\eps$.
\end{theorem}

To determine the limiting behavior of solutions to \eqref{ibvp_Blackstock_general} as $\eps \searrow 0$,  we state an analogous result to Theorem~\ref{Thm:Limit} on the continuity of the solution with respect to the memory kernel.
\begin{theorem}\label{Thm:cont_Black}
Let $\eps_1, \eps_2 \in (0, \beps)$. Under the assumptions of Theorem~\ref{Thm:Wellp_2ndorder_nonlocal_Blackstock}, for sufficiently small $T$, the following estimate holds:
\begin{equation}
\|\psi^{\eps_1}-\psi^{\eps_2}\|_{\textup{E}}\lesssim \|(\frakK_{\eps_1}-\frakK_{\eps_2})\Lconv1\|_{L^1(0,T)},
\end{equation}
where $\|\cdot\|_{\textup{E}}$ is defined in~\eqref{energy_norm}.
\end{theorem}
\begin{proof}
The difference $\opsi=\psi^{\eps_1}-\psi^{\eps_2}$ solves 
\begin{equation}
\begin{aligned}\label{eq:equation_frakeps_diff_Black}
&\opsi_{tt}-c^2 (1 - 2 \tk \psi_t^{\eps_1} ) \D \opsi-  \frakK_{\eps_1}* \D\opsi_{t}\\
=&\, \begin{multlined}[t] - 2\tk \opsi_t \D\psi^{\eps_2}- 2\ell \left( \nabla \psi^{\eps_1} \cdot \nabla \opsi_t +\nabla \opsi \cdot \nabla \psi^{\eps_2}_t \right)+(\frakK_{\eps_1}-\frakK_{\eps_2})* \Delta \psi_t^{\eps_2}
\end{multlined}
\end{aligned}
\end{equation}
with zero boundary and initial conditions.
As in the proof of Theorem~\ref{Thm:Limit}, we test \eqref{eq:equation_frakeps_diff_Black} by $\opsi_{t}$ and find ourselves in a similar setting to that of the lower-order estimate \eqref{first_est_selfm_Black}. We can then proceed similarly to the proof of Theorem~\ref{Thm:Limit} to treat the last two right-hand-side terms. The remaining term can be estimated as follows:
 \begin{equation}
 	\begin{aligned}
 		\left |-2\tk \int_0^t \intO \opsi_t^2 \D\psi^{\eps_2} \dxs \right| 
 		\lesssim\,  \|\D\psi ^{\eps_2}\|_{L^1(L^\infty(\Om))}\|\opsi_t\|_{L^\infty_t(L^2(\Om))}^2.
 	\end{aligned}
 \end{equation}
Note that thanks to Theorem~\ref{Thm:Wellp_2ndorder_nonlocal_Blackstock}, we have a uniform bound on \[
\|\D\psi ^{\eps_2}\|_{L^1(L^\infty(\Om))} \lesssim \|\D\psi ^{\eps_2}\|_{L^1(H^2(\Om))}.
\]
 We can thus proceed similarly to the proof of Theorem~\ref{Thm:Limit} to arrive at the desired statement statement.
\end{proof}

With this continuity result for the nonlocal Blackstock equation in hand, we can state the following counterparts of Corollary \ref{Corollary:Limit_epsK} and Proposition~\ref{Prop:Limit_a<=b}. 
\begin{corollary}\label{Cor:inv_limit_Black}
Under the assumptions of Theorems~\ref{Thm:Wellp_2ndorder_nonlocal_Blackstock} and~\ref{Thm:cont_Black}  with the kernel
\[
\frakKeps = \eps \frakK, \quad \eps \in (0, \beps)
\] 
satisfying assumptions \eqref{assumption1}, \eqref{assumption2}, and \eqref{assumption3_Black}, the family of solutions $\{\peps\}_{\eps \in (0, \beps)}$ of \eqref{ibvp_Blackstock_general}  converges in the energy norm to the solution $\psi$ of the initial-boundary value problem for the inviscid Blackstock equation:
\begin{equation}\label{ibvp_Black_limit}
\left \{	\begin{aligned} 
&\psi_{tt}-c^2  (1-2\tk\psi_t)\Delta \psi + 2\ell\, \nabla \psi\cdot\nabla \psi_t  = 0 \quad  &&\text{in } \Omega \times (0,T), \\
&\psi =0 \quad  &&\text{on } \partial \Omega \times (0,T),\\
&(\psi, \psi_t)=(\psi_0, \psi_1), \quad  &&\text{in }  \Omega \times \{0\},
\end{aligned} \right.
\end{equation}
at a linear rate	
\begin{equation}
	\|\peps-\psi\|_{\textup{E}}\lesssim \eps
	\quad \mbox{ as } \ \eps  \searrow 0.
\end{equation}
\end{corollary}
\begin{proof}
	Follows analogously to Corollary~\ref{Corollary:Limit_epsK}. The details are omitted.
\end{proof}
The thermal relaxation time limit result for specific families of Mittag-Leffler kernels is given by the following result.
\begin{proposition}
Let $\rt>0$ and $\delta>0$ be fixed, and let the assumptions of Theorems~\ref{Thm:Wellp_2ndorder_nonlocal_Blackstock} and~\ref{Thm:cont_Black}  hold. Consider the family of solutions $\{\peps\}_{\eps \in(0,\beps)}$ of \eqref{ibvp_Kuzn_general} with the kernel given by
\begin{equation} \label{kernel_form}
\begin{aligned}
\frakKeps(t)=&\,\delta\left(\frac{\rt}{\eps}\right)^{a-b}\frac{1}{\eps^b}t^{b-1}E_{a,b}\left(-\left(\frac{t}{\eps}\right)^a\right) \quad \text{where } \, 0 < a \leq b\leq 1.
\end{aligned}
\end{equation}
Then, the family $\{\peps\}_{\eps \in(0,\beps)}$
converges to the solution $\psi$ of 
\begin{equation}\label{ibvp_West_fractional_limit}
\left \{	\begin{aligned} 
&\psi_{tt}-c^2 (1-2\tk \psi_t)\Delta \psi -  \D\frakK_0\Lconv\psi_t + 2\ell \,\nabla\psi\cdot\nabla\psi_t = 0\quad  &&\text{in } \Omega \times (0,T), \\
&\psi =0 \quad  &&\text{on } \partial \Omega \times (0,T),\\
&(\psi, \psi_t)=(\psi_0, \psi_1), \quad  &&\text{in }  \Omega \times \{0\},
\end{aligned} \right.
\end{equation}
with the kernel $\frakK_0= \delta\rt^{a-b} g_{b-a}$ at the following rate
\begin{equation}
\|\peps-\psi\|_{\textup{E}}\lesssim \|(\frakKeps-\frakK_0)\Lconv1\|_{L^1(0,T)}  \sim\eps^a
\quad \mbox{ as } \ \eps  \searrow 0.
\end{equation}
\end{proposition}
\begin{proof}
	Follows analogously to Proposition~\ref{Prop:Limit_a<=b}. The details are omitted.
\end{proof}
\subsection{Comparison with the Westervelt equation}
In the context of nonlinear acoustics, the potential form Westervelt equation is obtained by setting $\ell=0$ and adjusting the value of the nonlinearity parameter $\tk$ in \eqref{abstract_wave_eq_Kuznetsov} (we denote the modified parameter by $\accentset{\approx}{k}$):
\begin{equation} \label{West_general_potential}
	\left \{	\begin{aligned} 
	&(1+2\accentset{\approx}{k} \pepst)\pepstt-c^2 \Delta \peps - \Delta \frakKeps * \peps = 0\quad  &&\text{in } \Omega \times (0,T), \\
	&\peps =0 \quad  &&\text{on } \partial \Omega \times (0,T),\\
	&(\peps, \pepst)=(\psi_0, \psi_1), \quad  &&\text{in }  \Omega \times \{0\}.
	\end{aligned}\right.
\end{equation} 

In this setting, the results of Section~\ref{Sec:Kuznetsov_prop} remain valid, but we can also obtain a bit more. Indeed, when $\ell=0$, we do not require \eqref{assumption2} as the bootstrap argument on page~\pageref{paragraph:bootstrap} is not needed. Then, under similar restrictions on final time, regularity and size of data to those of Theorem~\ref{Thm:Wellp_2ndorder_nonlocal_Kuzn}, we expect \eqref{West_general_potential} to be well-posed in $$W^{3,1}(0,T;\Honezero) \cap W^{2,1}(0,T;\Honetwo) \cap W^{1,\infty}(0,T;\Honethree).$$

It is possible to show, under the following additional assumption on $\rt$:
\begin{equation}\label{assumptiion_rt/eps_const}
	\rt = \rt(\eps) \quad \textrm{and } \left(\frac{\rt}{\eps}\right)^{a-b} = \rho^{a-b} = \textrm{constant},
\end{equation}
that the family of solutions $\{\peps\}_{\eps \in(0,\beps)}$ of  \eqref{ibvp_Kuzn_general} with the kernel given by
\begin{equation}
	\begin{aligned}
		\frakKeps(t)=&\,\delta\left(\frac{\rt}{\eps}\right)^{a-b}\frac{1}{\eps^b}t^{b-1}E_{a,b}\left(-\left(\frac{t}{\eps}\right)^a\right), \qquad 0<b<a \leq1,
	\end{aligned}
\end{equation}
converges in the energy norm ($\|\cdot\|_{\textup{E}}$) to the solution $\psi$ of the inviscid problem:
\begin{equation} 
	\left \{	\begin{aligned} 
		&(1+2\accentset{\approx}{k} \psi_t)\psi_{tt}-c^2 \Delta \psi = 0\quad  &&\text{in } \Omega \times (0,T), \\
		&\psi =0 \quad  &&\text{on } \partial \Omega \times (0,T),\\
		&(\psi, \psi_t)=(\psi_0, \psi_1), \quad  &&\text{in }  \Omega \times \{0\}
	\end{aligned}\right.
\end{equation} 
at the following rate:
\begin{equation} \label{linrate_Kuzn_tau2}
	\|\peps-\psi\|_{\textup{E}}\lesssim \|\frakKeps*1\|_{L^1(0,T)} \sim \eps ^{a-b} \quad \mbox{ as } \ \eps\searrow0.
\end{equation} 
This result is comparable to the one established in~\cite[Proposition 4.2]{kaltenbacher2023limting} where the pressure form of the Westervelt equation (see~\eqref{West_general}) was analyzed. We note, however, that the regularity requirements on initial data are higher here since the quasilinear coefficient contains a time derivative ($1+2\tk \pepst$ instead of $1+2\tilde{k}\ueps$), thus requiring higher-order energy arguments to control it.

\section*{Discussion}
In this work, we have investigated the nonlinear Kuznetsov and Blackstock equations with a general nonlocal dissipation term which encompasses the case of fractional-in-time damping.
We rigorously studied their $\eps$-uniform local well-posedness and their limiting behavior with respect to the parameter $\eps$. As we have seen, the limiting behavior and rate of convergence is influenced by the dependence of the memory kernel $\frakKeps$ on $\eps$. In particular, we have established limiting results for equations involving a class of Abel and Mittag-Leffler kernels. The difference between the Kuznetsov and Blackstock equations stems from the position of the nonlinearity which influences the manipulation of the equations and thus which terms need to be controlled. In particular, because the term $\pepstt$ appears only linearly in the Blackstock equation, its solution spaces need not be as regular in time as Kuznetsov's and the requirements on the kernel are weaker as well. However, the convergence rate to their respective limiting behavior is qualitatively similar for both equations.  \\
\indent The framework developed in this work and leading to Theorem~\ref{Thm:Limit} and Theorem~\ref{Thm:cont_Black} allows extending the limiting study to other parameters of interest, as long as the parameter-dependent kernels satisfy assumptions \eqref{assumption1}, \eqref{assumption2}, and \eqref{assumption3} (when considering Kuznetsov's nonlinearities) or the weaker \eqref{assumption3_Black} (when interested in Blackstock's). Among others, the limit as the fractional order $\alpha\nearrow1$ in the Abel kernels can be studied. 

\section*{Acknowledgments}
The work of the first author was supported by the Austrian Science Fund FWF under the grant DOC 78.
\bibliography{references}{}

\begin{thebibliography}{10}

\bibitem{agmon2010lectures}
{\sc S.~Agmon}, {\em Lectures on elliptic boundary value problems}, vol.~369,
  American Mathematical Soc., 2010.

\bibitem{baker2022numerical}
{\sc K.~Baker, L.~Banjai, and M.~Ptashnyk}, {\em Numerical analysis of a
  time-stepping method for the {W}estervelt equation with time-fractional
  damping}, arXiv preprint arXiv:2210.16349,  (2022).

\bibitem{blackstock1963approximate}
{\sc D.~T. Blackstock}, {\em Approximate equations governing finite-amplitude
  sound in thermoviscous fluids}, tech. rep., General Dynamics/Electronics
  Rochester NY, 1963.

\bibitem{compte1997generalized}
{\sc A.~Compte and R.~Metzler}, {\em The generalized {C}attaneo equation for
  the description of anomalous transport processes}, Journal of Physics A:
  Mathematical and General, 30 (1997), p.~7277.

\bibitem{dekkers2017cauchy}
{\sc A.~Dekkers and A.~Rozanova-Pierrat}, {\em Cauchy problem for the
  {K}uznetsov equation}, Discrete \& Continuous Dynamical Systems - {A}, 39
  (2019), pp.~277--307.

\bibitem{fritz2018well}
{\sc M.~Fritz, V.~Nikoli{\'c}, and B.~Wohlmuth}, {\em Well-posedness and
  numerical treatment of the {B}lackstock equation in nonlinear acoustics},
  Mathematical Models and Methods in Applied Sciences, 28 (2018),
  pp.~2557--2597.

\bibitem{grisvard2011elliptic}
{\sc P.~Grisvard}, {\em Elliptic problems in nonsmooth domains}, SIAM, 2011.

\bibitem{gurtin1968general}
{\sc M.~E. Gurtin and A.~C. Pipkin}, {\em A general theory of heat conduction
  with finite wave speeds}, Archive for Rational Mechanics and Analysis, 31
  (1968), pp.~113--126.

\bibitem{holm2019waves}
{\sc S.~Holm}, {\em Waves with Power-Law Attenuation}, Springer, 2019.

\bibitem{jin2021fractional}
{\sc B.~Jin}, {\em Fractional differential equations}, Springer, 2021.

\bibitem{jordan2014second}
{\sc P.~M. Jordan}, {\em Second-sound phenomena in inviscid, thermally relaxing
  gases}, Discrete \& Continuous Dynamical Systems-B, 19 (2014), p.~2189.

\bibitem{kaltenbacher2023limting}
{\sc B.~Kaltenbacher, M.~Meliani, and V.~Nikoli\'c}, {\em Limiting behavior of
  quasilinear wave equations with fractional-type dissipation}, arXiv preprint
  arXiv:2206.15245,  (2023).

\bibitem{kaltenbacher2022parabolic}
{\sc B.~Kaltenbacher and V.~Nikoli\'c}, {\em Parabolic approximation of
  quasilinear wave equations with applications in nonlinear acoustics}, SIAM
  Journal on Mathematical Analysis, 54 (2022), pp.~1593--1622.

\bibitem{kaltenbacher2022time}
{\sc B.~Kaltenbacher and V.~Nikoli{\'c}}, {\em Time-fractional
  {M}oore--{G}ibson--{T}hompson equations}, Mathematical Models and Methods in
  Applied Sciences, 32 (2022), pp.~965--1013.

\bibitem{kaltenbacher2022inverse}
{\sc B.~Kaltenbacher and W.~Rundell}, {\em On an inverse problem of nonlinear
  imaging with fractional damping}, Mathematics of Computation, 91 (2022),
  pp.~245--276.

\bibitem{kubica2020time}
{\sc A.~Kubica, K.~Ryszewska, and M.~Yamamoto}, {\em Time-fractional
  Differential Equations: A Theoretical Introduction}, Springer, 2020.

\bibitem{kuznetsov1971equations}
{\sc V.~P. Kuznetsov}, {\em Equations of nonlinear acoustics}, Soviet Physics:
  Acoustics, 16 (1970), pp.~467--470.

\bibitem{Wilke}
{\sc S.~Meyer and M.~Wilke}, {\em Global well-posedness and exponential
  stability for {K}uznetsov's equation in ${L}_p$-spaces}, Evolution Equations
  \& Control Theory, 2 (2013), pp.~365--378.

\bibitem{mizohata1993global}
{\sc K.~Mizohata and S.~Ukai}, {\em The global existence of small amplitude
  solutions to the nonlinear acoustic wave equation}, Journal of Mathematics of
  Kyoto University, 33 (1993), pp.~505--522.

\bibitem{nikolic2022time}
{\sc V.~Nikoli{\'c} and B.~Said-Houari}, {\em Time-weighted estimates for the
  {B}lackstock equation in nonlinear ultrasonics}, Journal of Evolution
  Equations, 23 (2023), p.~59.

\bibitem{podlubny1998fractional}
{\sc I.~Podlubny}, {\em Fractional differential equations: an introduction to
  fractional derivatives, fractional differential equations, to methods of
  their solution and some of their applications}, Elsevier, 1998.

\bibitem{povstenko2011fractional}
{\sc Y.~Povstenko}, {\em Fractional {C}attaneo-type equations and generalized
  thermoelasticity}, Journal of Thermal Stresses, 34 (2011), pp.~97--114.

\bibitem{tani2017mathematical}
{\sc A.~Tani}, {\em Mathematical analysis in nonlinear acoustics}, in AIP
  Conference Proceedings, vol.~1907, AIP Publishing LLC, 2017, p.~020003.

\bibitem{zhang2014time}
{\sc W.~Zhang, X.~Cai, and S.~Holm}, {\em Time-fractional heat equations and
  negative absolute temperatures}, Computers \& Mathematics with Applications,
  67 (2014), pp.~164--171.

\end{thebibliography}
\bibliographystyle{siam} 
\end{document}